\documentclass[11pt]{amsart}
\usepackage[bb=fourier,cal=pxtx]{mathalfa}
\usepackage{amssymb}
\usepackage{mathtools}		
\usepackage{mathabx}		  
\usepackage{indentfirst}	
\usepackage[normalem]{ulem}  
\usepackage{amsthm}
\usepackage{thmtools}
\usepackage{enumitem}		
\usepackage{multirow,bigdelim} 
\usepackage[backref=page,colorlinks=true]{hyperref}	
\usepackage[capitalise]{cleveref} 

\usepackage{tikz}
\usetikzlibrary{patterns}
\usepackage[font=small]{caption}

\usepackage{xcolor}
\hypersetup{
    colorlinks,
    linkcolor={red!50!black},
    citecolor={blue!50!black},
    urlcolor={blue!80!black}
}

\renewcommand*{\backref}[1]{}
\renewcommand*{\backrefalt}[4]{\quad \tiny 
  \ifcase #1 ({\color{red}\textbf{NOT CITED.}})%
  \or    (Cited on page~#2.)%
  \else   (Cited on pages~#2.)%
  \fi}

\makeatletter

\def\MRbibitem{\@ifnextchar[\my@lbibitem\my@bibitem}

\def\mybiblabel#1#2{\@biblabel{{\hyperref{http://www.ams.org/mathscinet-getitem?mr=#1}{}{}{#2}}}}

\def\myhyperanchor#1{\Hy@raisedlink{\hyper@anchorstart{cite.#1}\hyper@anchorend}}

\def\my@lbibitem[#1]#2#3#4\par{%
  \item[\mybiblabel{#2}{#1}\myhyperanchor{#3}\hfill]#4%
  \@ifundefined{ifbackrefparscan}{}{\BR@backref{#3}}%
  \if@filesw{\let\protect\noexpand\immediate
    \write\@auxout{\string\bibcite{#3}{#1}}}\fi\ignorespaces%
}

\def\my@bibitem#1#2#3\par{%
  \refstepcounter\@listctr
  \item[\mybiblabel{#1}{\the\value\@listctr}\myhyperanchor{#2}\hfill]#3%
  \@ifundefined{ifbackrefparscan}{}{\BR@backref{#2}}%
  \if@filesw\immediate\write\@auxout
    {\string\bibcite{#2}{\the\value\@listctr}}\fi\ignorespaces%
}

\makeatother

\declaretheorem[numberwithin=section]{theorem}
\declaretheorem[sibling=theorem]{lemma}
\declaretheorem[sibling=theorem]{corollary}
\declaretheorem[sibling=theorem]{proposition}
\declaretheorem[sibling=theorem]{conjecture}
\declaretheorem[sibling=theorem]{problem}
\declaretheorem[sibling=theorem,style=remark]{remark}

\hypersetup{bookmarksdepth = 3} 
\numberwithin{equation}{section}     

\Crefname{section}{Section}{Sections}
\Crefname{subsection}{Subsection}{Subsections}
\Crefname{subsubsection}{\S}{\S\S}
\Crefname{conjecture}{Conjecture}{Conjectures}

\setlist[enumerate,1]{label={\upshape(\alph*)},ref=\alph*}
\setlist[enumerate,2]{label={\upshape(\arabic*)},ref=\arabic*}

\newcommand{\arxiv}[1]{Preprint \href{http://arxiv.org/abs/#1}{arXiv:{#1}}}
\newcommand{\doi}[1]{\href{http://dx.doi.org/#1}{doi:{#1}}}

\newcommand{\R}{\mathbb{R}}
\newcommand{\Z}{\mathbb{Z}}

\newcommand{\T}{\mathbb{T}}
\newcommand{\B}{\mathbb{B}}

\newcommand{\cB}{\mathcal{B}}\newcommand{\cC}{\mathcal{C}}
\newcommand{\cF}{\mathcal{F}}
\newcommand{\cH}{\mathcal{H}}
\newcommand{\cL}{\mathcal{L}}

\newcommand{\cV}{\mathcal{V}}

\newcommand{\st}{\;\mathord{;}\;}

\newcommand{\GL}{\mathrm{GL}}

\newcommand{\Id}{\mathrm{Id}}

\DeclareMathOperator{\diam}{diam}
\newcommand{\Diff}{\mathrm{Diff}}
\newcommand{\id}{\mathrm{id}}
\newcommand{\dd}{\,\mathrm{d}}   
\newcommand{\diag}{\mathrm{diag}}
\newcommand{\wed}{\mathsf{\Lambda}}  
\newcommand{\len}{\mathrm{len}}

\newcommand{\reg}{r} 

\newcommand{\tribar}[1]{\mathopen{| {\kern -1.5pt} | {\kern -1.5pt} |} {#1} \mathclose{| {\kern -1.5pt} | {\kern -1.5pt} |}}
\newcommand{\bigtribar}[1]{\mathopen{\big|{\kern -1.5pt}\big|{\kern -1.5pt}\big|}{#1}\mathclose{\big|{\kern -1.5pt}\big|{\kern -1.5pt}\big|}}
\newcommand{\Bigtribar}[1]{\mathopen{\Big|{\kern -1.5pt}\Big|{\kern -1.5pt}\Big|}{#1}\mathclose{\Big|{\kern -1.5pt}\Big|{\kern -1.5pt}\Big|}}
\newcommand{\biggtribar}[1]{\mathopen{\bigg|{\kern -1.5pt}\bigg|{\kern -1.5pt}\bigg|}{#1}\mathclose{\bigg|{\kern -1.5pt}\bigg|{\kern -1.5pt}\bigg|}}
\newcommand{\biangle}[1]{\mathopen{\langle {\kern -1.7pt} \langle}{#1} \mathclose{\rangle {\kern -1.7pt} \rangle}}

\newcommand*\circled[1]{\tikz[baseline=(char.base)]{
    \node[shape=circle,draw,inner sep=1pt] (char) {\footnotesize{#1}};}}

\newcommand\coolrightbrace[2]{\left.\vphantom{\begin{matrix} #1 \end{matrix}}\right\}#2}

\renewcommand{\epsilon}{\varepsilon}
\renewcommand{\phi}{\varphi}
\renewcommand{\setminus}{\smallsetminus}

\begin{document}

\title{Flexibility of Lyapunov exponents}

\author[J.~Bochi]{J.~Bochi$^1$}
\email{\href{mailto:jairo.bochi@mat.uc.cl}{jairo.bochi@mat.uc.cl}}
\address{Facultad de Matem\'aticas, 
	Pontificia Universidad Cat\'olica de Chile, 
	Avenida Vicu\~na Mackenna 4860,
	Santiago, Chile }

\author[A.~Katok]{A.~Katok$^2$}

\author[F.~Rodriguez Hertz]{F.~Rodriguez Hertz$^3$}
\email{\href{mailto:fjrhertz@gmail.com}{fjrhertz@gmail.com}}
\address{Department of Mathematics,
        The Pennsylvania State University,
        University Park, PA 16802, 
        USA }

\date{August, 2019; revised March, 2021.}

\begin{abstract} 
We outline the flexibility program in smooth dynamics, focusing on flexibility of Lyapunov exponents for volume-preserving diffeomorphisms.
We prove flexibility results for Anosov diffeomorphisms admitting dominated splittings into one-dimensional bundles.
\end{abstract}

\begin{thanks}
{{}$^1$ Partially supported by Fondecyt 1180371 and Conicyt PIA ACT172001.
{}$^2$ Partially supported by NSF Grant DMS 1602409".
{}$^3$ Partially supported by NSF 1500947.
J.B.\ thanks the hospitality of Penn State University and Yeshiva University.
The authors thank the hospitality of Mathematisches Forschungsinstitut Oberwolfach.
}
\end{thanks}

\maketitle

\section{Introduction}

\subsection{The flexibility program}

Important attributes of smooth dynamical systems such as entropies and Lyapunov characteristic exponents with respect to a relevant invariant measure (i.e.\ a volume, an SRB measure, or a measure of maximal entropy) reflect  asymptotic behavior of orbits and with rare exceptions cannot be calculated in a closed form.  Exceptions are systems of algebraic origin, such as translations on homogeneous spaces and  affine maps on compact abelian groups and, in the case of topological entropy, structurally stable discrete time hyperbolic systems where topological entropy can be calculated using an algebraic or symbolic model. Beyond that there are few general relations for various classes of systems, in the form of equalities or inequalities, either involving  only dynamical characteristics themselves or relating those with other quantities coming from geometry, topology or analysis. Let us list some of those relations. Those marked with an asterisk are valid for topological dynamical systems on compact spaces; others require some smoothness assumptions. We refer to original sources only if no standard monograph or textbook exposition is available. 

\begin{itemize}
\item \emph{Variational principles for entropy*} \cite[Theorem~4.5.3]{KH} and \emph{pressure*} \cite[Theorem~20.2.4]{KH}.
\item \emph{Ruelle inequality} for $C^1$ systems \cite[Theorem~S.2.13]{KH}.
\item\emph{Pesin entropy formula} for $C^{1+\epsilon}$ systems preserving an absolutely continuous measure \cite[Theorem~10.4.1]{BP}.
\item \emph{Inequality between fundamental group growth and topological entropy*}
\cite[Theorem~8.1.1]{KH}.
\item\emph{Yomdin--Newhouse}  solution of the {Shub entropy conjecture} for $C^\infty$ systems \cite{Yomdin, Newhouse}.
\item For Anosov systems on infranilmanifolds Shub entropy inequality becomes equality (for the torus case, see \cite[Theorem~18.6.1]{KH}).
\item{Conformal inequality  for entropies}  for geodesic flows on manifolds of negative curvature \cite{K82}.
\end{itemize}

At a more basic level, preservation of a geometric structure imposes restrictions on dynamical invariants. For example, for a volume-preserving system the sum of Lyapunov characteristic exponents is zero, for a holomorphic system all exponents have even multiplicity, and for a symplectic map exponents come in pairs $\pm \lambda$.

The general paradigm of flexibility can be rather vaguely formulated as follows: 
\begin{quote}
{\bf($\frak F$)} \emph{Under properly understood general restrictions (like those listed or mentioned above),  within a fixed class of smooth dynamical systems  dynamical invariants take  arbitrary values.}
\end{quote}

In the context of smooth ergodic theory, one of the most natural flexibility problems concerns Lyapunov exponents for volume-preserving systems with respect to the volume measure.
We mostly restrict our discussion to the classical discrete time invertible dynamical systems, i.e.\ actions of $\Z$. The continuous time case in some key situations follows directly from the discrete one via the suspension construction, in the others can be treated in a parallel way  and in certain respects it is easier since the homotopy restrictions (see below) do not appear. 

The case of multidimensional time is very different.  
There the phenomenon of rigidity that in a sense is complementary to flexibility is prevalent: see e.g.~\cite{KRH}.

Previously to the appearance of this paper, some instances of flexibility have been investigated by Hu, M.~Jiang, and Y.~Jiang \cite{HuJJ_older,HuJJ}, Erchenko \cite{Erch}, Erchenko and Katok \cite{ErchK}, Barthelm\'e and Erchenko \cite{BErch1, BErch2}.


\subsection{General conservative diffeomorphisms}\label{ss.general}

Let $M$ be a smooth compact connected 
manifold of dimension $d \ge 2$, with or without boundary, $f \colon M\to M$ be a diffeomorphism of $M$ and $\mu$ an $f$-invariant ergodic Borel probability measure. 
By the Oseledets Multiplicative Ergodic Theorem, the limits
\begin{equation}
\lim_{n \to \pm\infty} \frac{1}{n}\log\Big( \text{$i$-th singular value of } Df^n(x) \Big) 
\end{equation}
exist  and hence are constant $\mu$-almost everywhere. 
They are called  \emph{Lyapunov  characteristic exponents} or often simply \emph{Lyapunov exponents} of $f$ with respect to $\mu$ and are denoted by $\lambda_{1,\mu}(f) \ge \cdots \ge \lambda_{d,\mu}(f)$.   For the full Oseledets Theorem (which also describes the growth of tangent vectors), see e.g.\  \cite{LArnold,BP}.  The \emph{Lyapunov spectrum} is defined as the vector
\begin{equation}
\boldsymbol{\lambda}_\mu(f) \coloneqq \big( \lambda_{1,\mu}(f),\dots,\lambda_{d,\mu}(f) \big) \, .
\end{equation}
We say that this spectrum is \emph{simple} if none of these numbers is repeated.

Let $m$ be a smooth volume measure, normalized so that $m(M)=1$. 
The particular choice is not important, since for any pair of  such measures, there exists a diffeomorphism taking one to the other \cite{Moser}, \cite[Theorem 5.1.27]{KH}. 
Given $\reg\in \{1,2,\dots,\infty\}$, let $\Diff_m^\reg(M)$ denote the set of $m$-preserving (also called \emph{conservative}) diffeomorphisms $f\colon M \to M$ of class $C^\reg$. 
We will discuss the case when $f$ is ergodic with respect to $m$;
for simplicity we write $\lambda_i(f) = \lambda_{i,m}(f)$, $\boldsymbol{\lambda}(f) = \boldsymbol{\lambda}_m(f)$.
We always have $\sum_{i=1}^d \lambda_i(f) = 0$.

\medskip

Now we formulate  and discuss several representative  questions concerning flexibility of Lyapunov exponents  for general conservative diffeomorphisms.

\begin{conjecture}[Weak flexibility -- general]\label{conj.general_weak} 
Given  any list of numbers $\xi_{1} \ge \dots \ge\xi_{d}$ with $\sum_{i=1}^d\xi_i=0$, there exists an ergodic diffeomorphism $f\in\Diff_m^\infty(M)$ such that 
$\boldsymbol{\lambda}(f) = (\xi_{1}, \dots, \xi_{d})$.
\end{conjecture}

\begin{conjecture}[Strong flexibility -- general]\label{conj.general_strong} 
Given a connected component $\cC\subseteq\Diff_m^\infty(M)$ and any list of numbers $\xi_{1} \ge  \dots \ge\xi_{d}$ with $\sum_{i=1}^d\xi_i=0$, there exists an ergodic diffeomorphism $f\in\cC$ such that 
$\boldsymbol{\lambda}(f) = (\xi_{1}, \dots, \xi_{d})$.
\end{conjecture}

If all exponents are equal to zero then \cref{conj.general_weak} is known; this has been proved long ago \cite{AnosovK, Anosov}.
In this case \cref{conj.general_strong} holds for the identity component provided that the dimension is at least $3$.\footnote{Existence of ergodic diffeomorphisms with zero exponents on any manifold with a non-trivial action of the circle $S^1$ including the  $2$-disc $\mathbb D^2$, $2$-sphere $\mathbb S^2$, the annulus, and the Klein bottle, has been established in the  paper \cite{AnosovK}, which can be viewed as the earliest work on flexibility. However in the case of $\mathbb D^2$ in those examples the action on the boundary is an irrational rotation with a Liouvillean rotation number. Existence of  zero entropy ergodic examples that are identity or have a rational rotation number on the boundary is an open and probably very difficult question.}

On an opposite direction, the existence of conservative ergodic (actually Bernoulli) smooth diffeomorphisms \emph{without} zero Lyapunov exponents on any manifold was established by Dolgopyat and Pesin \cite{DolgoPesin} (the $2$-dimensional case was settled earlier \cite{K79}). 
These examples are homotopic to the identity.

\medskip

In this paper we will not attack \cref{conj.general_weak,conj.general_strong} directly.
Instead, we will establish flexibility results for a particular and more tractable class of systems, namely Anosov diffeomorphisms admitting simple dominated splitting. 
Nevertheless, we believe that our methods (combined with techniques from the aforementioned works) should provide the basis for an approach on the conjectures, at least under some restrictions.

\subsection{The Anosov case}\label{ss.Anosov}

Anosov systems represent a natural class for the flexibility analysis.  
We work with conservative Anosov diffeomophisms which are at least $C^2$; then, by a classical theorem of Anosov and Sinai, the volume measure $m$ is ergodic.

All known Anosov diffeomorphisms are topologically conjugate to automorphisms of infranilmanifolds that include tori and nilmanifolds as special cases. Hence the metric entropy with respect to invariant volume (equal to the sum of positive Lyapunov exponents) does not exceed the sum of positive Lyapunov exponents for the corresponding automorphism that is determined  by  induced automorphism of the fundamental group. The main  flexibility question is whether this is the only restriction.


In order to simplify the notation we restrict our discussion to the torus case.
Let $L\in\GL(d,\Z)$ and assume that $L$ is hyperbolic, i.e.\ the absolute values of all of its   eigenvalues are different from one. The matrix $L$ determines the automorphism $F_L$ of the torus $\T^d \coloneqq \R^d/\Z^d$, which is a conservative Anosov diffeomorphism.

Every Anosov diffeomorphism $f$ of $\T^d$ (conservative or not) 
is homotopic and, moreover, topologically conjugate via a homeomorphism isotopic to identity, to an automorphism $F_L$, where $L$ is a hyperbolic matrix \cite[Theorem~18.6.1]{KH}.
In fact, $L$ is the matrix of the automorphism induced by $f$ on the fundamental group of $\T^d$, which is naturally isomorphic to $\Z^d$.\footnote{However, for large enough $d$ is not always true that there is an homotopy between $f$ and $F_L$ consisting of Anosov diffeomorphisms: see \cite{FG}.}

Given a hyperbolic matrix $L\in\GL(d,\Z)$, the Lyapunov spectrum of the automorphism $\boldsymbol{\lambda}(F_L)$ is the vector $\boldsymbol{\lambda}(L)$ whose entries $\lambda_1(L) \ge \cdots \ge \lambda_d(L)$ are the logarithms of the absolute values of the eigenvalues of $L$, repeated according to multiplicity.
The number $u = u(L)$ of positive elements in this list is called the \emph{unstable index} of $L$; so $\lambda_u(L)>0>\lambda_{u+1}(L)$. 
The quantity $\sum_{i=1}^u\lambda_i(L)$  
is equal to both topological entropy $h_\mathrm{top}(F_L)$ and to the metric entropy $h_m(F_L)$ with respect to Lebesgue measure $m$ on $\T^d$.
Therefore,
for any conservative Anosov $C^{1+\epsilon}$-diffeomorphism $f \colon \T^d \to \T^d$ homotopic to $F_L$
we have:
\begin{equation}\label{e.h_condition_proof}
\sum_{i=1}^u\lambda_i(f) = h_m(f) \leq h_\mathrm{top}(f) = h_\mathrm{top}(F_L) = \sum_{i=1}^u\lambda_i(L) \, ,
\end{equation}
using Pesin's formula, the variational principle, and the above-mentioned topological conjugacy.\footnote{In reality, the inequality $\sum_{i=1}^u\xi_i\le\sum_{i=1}^u\lambda_i(L)$ also holds when $f$ is only $C^1$; indeed it follows from the $C^{1+\epsilon}$ case using $C^1$-continuity of the right-hand side and Avila's regularization \cite{Avila}.}
Are there other restrictions on the spectrum of $f$?
We pose the following:

\begin{problem}[Strong flexibility -- Anosov]\label{probl.Anosov_strong}
Let $L\in\GL(d,\Z)$ be a hyperbolic matrix, and let $u$ be its unstable index. 
Given any list of  numbers $\xi_1 \ge \cdots \ge \xi_u > 0 > \xi_{u+1} \ge \cdots \ge \xi_d$ such that 
\begin{equation}\label{e.h_condition}
\sum_{i=1}^d\xi_i=0
\quad\text{and}\quad
\sum_{i=1}^u\xi_i\le\sum_{i=1}^u\lambda_i(L) \, ,
\end{equation}
does there exist a conservative Anosov diffeomorphism $f$ homotopic (and hence topologically conjugate) to $F_L$ such that
$\boldsymbol{\lambda}(f) = (\xi_{1}, \dots, \xi_{d})$?
\end{problem}

Regularity of $f$ may vary but it does not seem likely that the answer depends on regularity, at least above $C^1$.\footnote{It may be more challenging to make some exponents equal in regularity above $C^1$.}

For $d=2$ (and so $u=1$), the problem reduces to existence of Anosov diffeomorphisms on $\T^2$ with any positive value of metric entropy below the topological entropy. 
Here the answer is positive.
It is not difficult to produce such examples even in the real-analytic category   by a fairly straightforward global twist construction.\footnote{See \href{https://jairobochi.wordpress.com/2021/03/25/flexibility_dim_2/}{https://jairobochi.wordpress.com/2021/03/25/flexibility\_dim\_2/}} 
Existence of $C^\infty$ examples also follows from our \cref{t.majorize} below.

In its weak version, i.e.\ without considerations about homotopy, the flexibility problem is likely to have a positive solution:

\begin{conjecture}[Weak flexibility -- Anosov]\label{conj.Anosov_weak}
Given any list of nonzero numbers $\xi_{1} \ge \dots \ge \xi_{d}$ such that 
$\sum_{i=1}^d\xi_i=0$, there exists  an Anosov diffeomorphism of $\T^d$ such that
$\boldsymbol{\lambda}(f) = (\xi_{1}, \dots, \xi_{d})$.
\end{conjecture}

The case of this conjecture with strict inequalities easily follows from our main result: see \cref{c.easy} below.

\subsection{Dominated splittings}\label{ss.domination}

Given a diffeomorphism $f \colon M \to M$, a $Df$-invariant splitting $T M = E_1 \oplus \cdots \oplus E_k$ into bundles of constant dimension is called \emph{dominated} if each of the bundles dominates the next.
This means that given a Riemannian metric, there exists $n_0 \ge 1$ such that
for every $x \in M$ and all unit vectors $v_1 \in E_1(x)$, \dots, $v_k \in E_k(x)$, we have 
\begin{equation}
\| Df^{n_0}(x) v_1 \| > \| Df^{n_0}(x) v_2 \| > \cdots > \| Df^{n_0}(x) v_k \| \, .
\end{equation}
It is always possible to find an ``adapted'' Riemannian metric for which ${n_0 = 1}$: see \cite{Gourmelon}.
Dominated splittings are automatically continuous, and their existence is a $C^1$-open condition; see e.g.\ \cite[\S~B.1]{BDV} for these and other properties.
We say that  a dominated splitting is \emph{simple}  
if all the subbundles $E_j$ are one-dimensional (and so $k=d$).
In this case, the Oseledets splitting coincides with $E_1 \oplus \cdots \oplus E_d$ almost everywhere, the Lyapunov spectrum is simple, and the Lyapunov exponents with respect to invariant volume $m$ are given by integrals:
\begin{equation}\label{e.exp_integrals}
\lambda_j(f) = \int_M \log \|Df |_{E_j}\| \dd m \, . 
\end{equation}
In particular, in the class of diffeomorphisms admitting a simple dominated splitting, the Lyapunov exponents depend continuously on the dynamics, 
and, therefore, the flexibility analysis becomes more manageable. 
On the other hand, in the absence of domination, small perturbations (with respect to the $C^1$ topology) of the dynamics may have a large effect on the Lyapunov spectrum
and even send all Lyapunov exponents to zero \cite{Bochi_ETDS,BochiViana} (but the $C^2$ norm of such a perturbation generally explodes \cite{LMY}). 

Existence of a dominated splitting also imposes restrictions on the topology of the manifold.

\subsection{Formulation of results}

Let us recall the classical notion of \emph{majorization},  
which has a wide range of applications (see e.g.\ \cite{MOA}).\footnote{See \cite{BoBo} for another instance where majorization plays a role in the perturbation of Lyapunov exponents.}

Suppose that $\boldsymbol{\xi} = (\xi_1,\dots,\xi_d)$ and $\boldsymbol{\eta}=(\eta_1,\dots,\eta_d)$ are \emph{ordered} vectors in $\R^d$, in the sense that
$\xi_1 \ge \cdots \ge \xi_d$ and $\eta_1 \ge \cdots \ge \eta_d$.
We say that $\boldsymbol{\xi}$ \emph{majorizes} $\boldsymbol{\eta}$ (or $\boldsymbol\eta$ \emph{is majorized by} $\boldsymbol\xi$)
if the following conditions hold:
\begin{align}
\xi_1 + \dots + \xi_j &\ge \eta_1 + \dots + \eta_j \quad \text{for all $j \in \{1,\dots, d-1\}$, and} \label{e.major1} \\
\xi_1 + \dots + \xi_d &= \eta_1 + \dots + \eta_d. \label{e.major2}
\end{align}
This is denoted by $\boldsymbol\xi \succcurlyeq \boldsymbol\eta$ (or $\boldsymbol\eta \preccurlyeq \boldsymbol\xi$),
and defines a partial order among ordered vectors.\footnote{Though we work with ordered vectors only, let us note that the partial order can be extended to all of $\R^d$ by declaring that $\boldsymbol\xi \succcurlyeq \boldsymbol\eta$ whenever $\boldsymbol\xi$ and $\boldsymbol\eta$ can be obtained from one another by permuting entries.}
If all the inequalities \eqref{e.major1} are strict (and the equality \eqref{e.major2} holds) then we say that $\boldsymbol{\xi}$ \emph{strictly majorizes} $\boldsymbol{\eta}$ (or $\boldsymbol{\eta}$ \emph{is strictly majorized by} $\boldsymbol{\eta}$),
and denote this by $\boldsymbol\xi \succ \boldsymbol\eta$  (or $\boldsymbol\eta \prec \boldsymbol\xi$).

Intuitively, 
$\boldsymbol\xi \succcurlyeq \boldsymbol\eta$ means that the entries of $\boldsymbol\eta$ are obtained from those of $\boldsymbol\xi$ by a process of ``mixing''.
Let us state this precisely:
If $\boldsymbol{\xi}$ majorizes~$\boldsymbol{\eta}$ then there exists a doubly-stochastic $d\times d$ matrix $P$ such that $\boldsymbol{\eta} = P \boldsymbol{\xi}$; conversely, given an ordered vector $\boldsymbol{\xi}$ and a doubly-stochastic matrix $P$, the vector obtained by reordering the entries of $P \boldsymbol{\xi}$ is majorized by $\boldsymbol\xi$
-- see \cite[Theorem~B.2]{MOA}.

\medskip

We now state the main result of this paper. 
Recall that $M$ is a smooth compact connected  manifold of dimension $d \ge 2$, and $m$ is a smooth volume measure, normalized so that $m(M)=1$; note that we do not assume that $M$ is a torus (nor even an infranilmanifold).
The \emph{unstable index} of an Anosov diffeomorphism is the dimension of its unstable bundle.

\begin{theorem}\label{t.majorize}
Let $\reg\in \{2,3,\dots,\infty\}$, and let $f \in \Diff_m^\reg(M)$ be a conservative Anosov $C^\reg$-diffeomorphism 
with simple dominated splitting.
Let $\boldsymbol{\xi} \in \R^d$ be such that: 
\begin{enumerate}
\item \label{i.spec_12}
$\xi_1 > \cdots > \xi_u > 0 > \xi_{u+1} > \cdots > \xi_d$, where $u$ is the unstable index of $f$;
\item\label{i.spec_3} 
$\boldsymbol{\xi} \prec \boldsymbol{\lambda}(f)$, that is, $\boldsymbol{\xi}$ is strictly majorized by $\boldsymbol{\lambda}(f)$.
\end{enumerate}
Then there is a continuous path $(f_t)_{t \in [0,1]}$ in $\Diff_m^\reg(M)$ such that:
\begin{itemize}
\item $f_0 = f$;
\item each $f_t$ is Anosov with simple dominated splitting;
\item $\boldsymbol{\lambda}(f_1) = \boldsymbol{\xi}$.
\end{itemize}
\end{theorem}

\begin{corollary}[Anosov diffeomorphisms display all hyperbolic simple Lyapunov spectra]\label{c.easy}
Given any list of nonzero numbers $\xi_{1} > \dots > \xi_{d}$ whose sum is equal to $0$, there exists a conservative  Anosov $C^\infty$ diffeomorphism of $\T^d$  with simple dominated splitting such that $\boldsymbol{\lambda}(f) = (\xi_{1}, \dots, \xi_{d})$. 
\end{corollary}

\cref{c.easy} is obtained as follows: first, we take an Anosov linear automorphism whose spectrum is simple and ``large'' with respect to the majorization partial order; then, \cref{t.majorize} allows us to deform the linear automorphism and obtain a conservative Anosov diffeomorphism with the desired Lyapunov spectrum.  (See \cref{ss.easy_proof} for full details.)

Note that as a consequence of \cref{c.easy} we obtain a positive solution of the general weak flexibility \cref{conj.general_weak} on tori for simple spectra: if the desired spectrum contains $0$ then we just take the product $f = g \times R_\theta$ of an appropriate Anosov map $g$ on $\T^{d-1}$ and an irrational rotation $R_\theta$ on~$\T$; this $f$ is ergodic since $g$ is mixing.

\medskip

While condition \eqref{i.spec_3} in \cref{t.majorize} asks for \emph{strict} majorization, there are specific situations where this requirement can be relaxed to ordinary majorization.
This is demonstrated by the next theorem, which also shows that the majorization condition is indeed necessary:

\begin{theorem}\label{t.T3}
Let $F_L$ be an Anosov linear automorphism of $\T^3$ with simple Lyapunov spectrum and unstable index $u$. For any given $\boldsymbol{\xi} = (\xi_1,\xi_2,\xi_3) \in \R^3$ there exists an Anosov diffeomorphism $f \in \Diff_m^\infty(\T^3)$ homotopic (and hence topologically conjugate)
to $F_L$ and with simple dominated splitting such that $\boldsymbol{\lambda}(f) = \boldsymbol{\xi}$ if and only if 
\begin{equation}
\xi_1>\xi_2>\xi_3\, , \quad
\xi_u>0>\xi_{u+1}\, , \quad \text{and} \quad
\boldsymbol{\xi} \preccurlyeq \boldsymbol{\lambda}(L) \, .
\end{equation}
Furthermore, one can choose a homotopy between $F_L$ and $f$ consisting of conservative smooth Anosov diffeomorphisms with simple dominated splitting. 
\end{theorem}

In $d=3$, the condition $\boldsymbol{\xi} \preccurlyeq \boldsymbol{\lambda}(L)$ is strictly stronger than the ``entropy condition'' \eqref{e.h_condition}. Therefore, if \cref{probl.Anosov_strong} has a positive solution, it necessarily involves Anosov diffeomorphisms without a simple dominated splitting, even in the case of simple Lyapunov spectra.
The existence of a single conservative Anosov diffeomorphism $f \colon \T^3 \to \T^3$ whose spectrum $\boldsymbol{\xi} = \boldsymbol{\lambda}(f)$ is not majorized by $\boldsymbol{\lambda}(L)$ is already a very interesting question.

\medskip

Let us note that Hu, M.~Jiang, and Y.~Jiang \cite{HuJJ} have constructed deformations of conservative Anosov diffeomorphisms and of conservative expanding endomorphisms having arbitrarily small metric entropy. In the case of diffeomorphisms with dominated splittings, their result follows from \cref{t.majorize}. Their construction is very different from ours.

\subsection{Comments on the proofs}

Let us summarize the ideas of the proof of \cref{t.majorize}.
Motivated by the work of Shub and Wilkinson \cite{ShubWilk}, Baraviera and Bonatti \cite{BarBon} have proved the following ``local flexibility'' result: 
given a conservative stably ergodic partially hyperbolic diffeomorphism, one can perturb it so that the sum of central Lyapunov exponents becomes different from zero. Their idea was to perturb the diffeomorphism on a small ball around a non-periodic point (so to avoid fast returns) by rotating on a center-unstable plane so that the central bundle borrows some expansion from the unstable bundle. 
Actually, their argument allows to slightly mix Lyapunov exponents in any pair of consecutive bundles in a dominated splitting, while the Lyapunov exponents in the other bundles move extremely little.\footnote{Furthermore, the main result of Baraviera and Bonatti \cite{BarBon} also applies to non-ergodic diffeomorphisms, in which case it controls only \emph{averaged} Lyapunov exponents. Avila, Crovisier, and Wilkinson \cite{ACW} perfected the method in order to control pointwise Lyapunov exponents directly without assuming ergodicity. These subtleties do not concern us because our dynamics are ergodic.} So it is conceivable that with a sequence of Baraviera--Bonatti perturbations one could mix Lyapunov exponents basically at will. Though this idea is ultimately correct, several difficulties need to be overcame in order to turn it into a proof of \cref{t.majorize}. First, how can we ensure that the effect of the perturbations is \emph{not too weak}? Second, how can we obtain a prescribed Lyapunov spectrum \emph{exactly}?

Let us discuss how to overcome the first difficulty.
Instead of using a single ball as the support of a perturbation, we select several small balls whose union has a large first return time but non-negligible measure: this is done with a standard tower construction (in the style of \cite{Bochi_ETDS,BochiViana}, for instance). In each of these balls one composes with the same ``model'' perturbation of the identity that rotates the appropriate plane. However, this trick by itself is not sufficient to conclude. If the domination or the hyperbolicity gets weak, the rotations should be smaller and their effect on the Lyapunov exponents are also small. The solution is \emph{to select carefully not only the location of the balls but also their shape}. This is done using a special adapted Riemannian metric such that on the one hand  domination and hyperbolicity are seen in a single iterate (as in \cite{Gourmelon}), but on the other hand, for a large proportion of points with respect to the reference measure $m$, the rates of expansion on a single iterate with respect to the adapted metric are very close to the Lyapunov exponents. We only perform the perturbations on balls around those good points. In this way we can ensure that the effect of the perturbation on the Lyapunov exponents is considerable. In this regard, we also remark that we have no bound for the $C^1$ size of the deformation $(f_t)_{t \in [0,1]}$ that we eventually construct in \cref{t.majorize}. Therefore, we need to be careful with the quantifiers to ensure some effectiveness of the perturbation without knowing which diffeomorphism we are perturbing.

Concerning the second difficulty, we note that the type of perturbations sketched above always has some small ``noisy'' effect on the Lyapunov exponents that cannot be made exactly zero (except if some of the invariant subbundles are smoothly integrable). 
To resolve this, we define our perturbations depending on several parameters, allowing us to move the Lyapunov spectrum in all directions. By topological reasons, this eventually permits us to obtain open sets of spectra. Now, in order to construct these multiparametric perturbations, in principle one could try to compose several Baraviera--Bonatti-like perturbations that mix different pairs of Lyapunov exponents. This idea turns out to be impractical, essentially because one would need to consider adapted metrics and towers depending on parameters. Luckily, it is possible to define a new type of multiparametric model perturbation that includes Baraviera--Bonatti perturbations as a particular case, freeing us from the trouble of working separately with each pair of exponents.

\medskip

Now let us outline the proof of \cref{t.T3}.
In order to manipulate, say, the first two Lyapunov exponents while keeping the third one unchanged we follow the same strategy as in the proof of \cref{t.majorize}, but using Baraviera--Bonatti perturbations that preserve the center-unstable foliation. This is possible because we start with the automorphism $F_L$ for which this foliation is smooth.
In the converse direction, suppose that $f$ is a conservative Anosov diffeomorphism of $\T^3$ homotopic to $F_L$ whose Lyapunov spectrum is simple but is \emph{not} majorized by the spectrum of $L$.
For example, consider the case $u=2$ and $\lambda_1(f)>\lambda_1(L)$.
If $f$ had a simple dominated splitting then the exponential growth rate of the strong unstable foliation of $f$ would be bigger than the corresponding rate for $F_L$, and this would contradict the quasi-isometric property of this foliation, obtained by Brin, Burago, and Ivanov \cite{BBI}.

\subsection{Further directions of research and other instances of the flexibility program}\label{ss.future}

\subsubsection{Extensions of our methods}\label{sss.extensions}


Hyperbolicity is not fundamental to our constructions; as in Baraviera--Bonatti \cite{BarBon}, domination is the important keyword. 
In fact, this perturbation method seems much more versatile, and should apply if domination is allowed to degenerate in a controlled way on a certain singular set, as it is the case in \cite{K79,DolgoPesin}.
Therefore, the conjectures of \cref{ss.general} seem approachable, at least in some cases.

\medskip

Let us briefly discuss what happens when we approach the boundary of set of allowable spectra $\boldsymbol{\xi}$ in \cref{t.majorize}.
We consider the principal case $f=F_L$. 
Let $u$ be  the unstable  index of $F_L$.
Consider the set of $\boldsymbol{\xi}$ that meet conditions \eqref{i.spec_12} and \eqref{i.spec_3} in the \lcnamecref{t.majorize}.
There are three types of components of the boundary:
\begin{enumerate}
\item $\xi_u=0$ or $\xi_{u+1} = 0$. One can carry our construction across either of those. The resulting map of course is not Anosov anymore but partially hyperbolic with one-dimensional central bundle.\footnote{A similar construction appears in \cite{PonceTah}, but starting with an automorphism $F_L$ of $\T^3$ whose central exponent is already close to $0$.}

\item $\sum_{i=1}^k\xi_i=\sum_{i=1}^k\lambda_i(L)$ for some $k\in\{1,\dots, d-1\}$. 
Cases $k=1$ and $k=d-1$ are feasible: as in the proof of the ``if'' part of \cref{t.T3}, the perturbations preserve codimension one $F_L$-invariant foliations.
For other values of $k$ we get into difficulties: our construction is iterative, but after the first step the foliation that needs to be preserved would no longer be smooth.  

\item $\xi_i=\xi_{i+1}$ for some $i \in \{1,\dots, d-1\}$. Our construction degenerates since the amount of domination decreases and the Lyapunov metric explodes. 

\end{enumerate}

Related to the last point, in order to realize non-simple Lyapunov spectrum (e.g., to prove \cref{conj.Anosov_weak}), one should be able to manipulate Lyapunov exponents at least in some cases of non-simple dominated splittings. But then formula \eqref{e.exp_integrals} does not apply, so finer methods would be required.

\subsubsection{Explicit constructions and bounds}
As mentioned before, in dimension~$2$ there exists a simple construction of conservative Anosov diffeomorphisms with prescribed Lyapunov spectra. 
It would be interesting to have such explicit constructions on higher-dimensional tori as well. 

Here is an exciting general question\footnote{We thank a referee for suggesting it.}: What would be the effect of bounds on the $C^r$ norms on the flexibility results? Our constructions are certainly very ``expensive'' in terms of $C^2$ norms, and probably in terms of $C^1$ norms as well (since do not have estimates on the eccentricities of the Lyapunov balls).

\subsubsection{Other measures}
In other settings, the invariant measure (or measures) one is interested in is not necessarily fixed in the class of dynamical systems under consideration, but varies with the dynamics itself. The prototypical example consists of equilibrium states of sufficiently hyperbolic dynamics with respect to relevant potentials. 
The flexibility paradigm then applies.
See \cite{Erch} for a result in this direction, where expanding maps on the circle are considered.

\subsubsection{Symplectic systems}
The problems that we have posed in the volume-preserving setting have symplectic counterparts, where symmetry of the spectrum appears as an extra requirement. 
It is plausible that our methods can be adapted to the symplectic setting, but not in a straightforward way.

\subsubsection{Flows}

Structural stability for flows does not imply topological conjugacy,  so topological entropy becomes a free parameter even in the uniformly hyperbolic case. Thus, for conservative Anosov  flows on three-dimensional manifolds the   basic flexibility problem involves realization of arbitrary  pairs of numbers as values for topological and metric entropy subject only to the variational inequality. This problem can be solved fairly easily for suspensions of $\T^2$  and unit tangent bundles of surfaces of genus $g\ge 2$ using time changes of homogeneous models  but becomes more interesting (probably still tractable)  for exotic Anosov flows where homogeneous models are not available. 

The problem  becomes really interesting  in the standard setting  when one considers special classes of Anosov flows. The prime example here  is provided  by geodesic flows on compact surfaces of negative curvature. In this case, a natural normalization is available by fixing the total surface area. The  variational inequality is strengthened  by the conformal inequality \cite{K82} and  possible values of the metric entropy $h_m$ and topological entropy $h_\mathrm{top}$ for Riemannian metrics of negative non-constant curvature are restricted to: 
\[
h_m<\left(\frac{4\pi(g-1)}{V}\right)^{\frac{1}{2}}<h_\mathrm{top},
\]
where $g$ is the genus of the surface and $V$ is its total area. 
For any constant curvature metric those inequalities become equalities. The fact that equality  of the metric and topological entropy implies constant curvature is a prototype case of \emph{rigidity} that is discussed below. 
The flexibility problem in this setting is solved in \cite{ErchK}. 
The solution uses methods which are totally different from those of the present work. 
For flexibility results where the conformal class is fixed (and also taking other invariant quantities into account), see \cite{BErch1,BErch2}.
More problems are posed in \cite[Sec.~4]{ErchK} and \cite[Sec.~7]{BErch2}.
 
The  higher dimensional case is wide open. 
One of the difficulties  of dealing with  geodesic flows  on higher dimensional negatively curved manifolds is that   algebraic models have  non-simple Lyapunov spectrum (in fact, either one or two Lyapunov exponents are of the same sign and full splittings never exist).

\subsubsection{Flexibility and rigidity} 

In the context of  conservative smooth dynamical systems the strongest natural equivalence relation is smooth conjugacy.  For symplectic systems it is  similarly a symplectic conjugacy, for geodesic flows --  isometry of the underlying metrics, and so on. 
A general  phenomenon of rigidity in this context can be described as follows:
\begin{quote}
{\bf($\frak R$)} \emph{Values of finitely many invariants determine the system either locally, i.e.\ in a certain neighborhood of a ``model'', or globally within an a priori defined class of systems.}
\end{quote}

The space of equivalence classes at best  can be given some natural infinite-dimensional structure and at worst is ``wild''. This is proven in a number of situations but in general it has to be viewed as  a paradigmatic statement, not a theorem, and, beyond $C^1$ case,  almost every meaningful general question is open. Thus, rigidity should be quite rare since it should appear for very special values of invariants.\footnote{Nevertheless, for zero entropy systems, local rigidity  of toral translations appears in the context of KAM theory and global rigidity for circle diffeomorphisms with Diophantine rotation number. Those situations do not concern us here.}

A natural question in our context is whether particular values of Lyapunov exponents (for conservative systems, or symplectic systems, or geodesic flows) imply rigidity. In agreement with the general  flexibility paradigm one may expect this to happen at the extreme allowable values of exponents. This tends to be true in low dimension, not true in full  generality and probable in a number of interesting situations. 
Several recent results fit this pattern: see the papers \cite{MicenaTah,SaghinYang,GKS} (concerning conservative Anosov diffeomorphisms, mostly) and \cite{Butler1,Butler2} (concerning geodesic flows).

\subsection{Organization of the paper}

The rest of this paper is organized as follows. 
In \cref{s.reduction} we state the technical \cref{p.central}, which is a local multi-parametric version of \cref{t.majorize}, and we show how it implies the theorem.
In \cref{s.adapted} we construct the adapted metrics mentioned above.
In \cref{s.damping} we use them 
to define damping perturbations; 
these are actually ``large perturbations'', but we show that their results  are still Anosov under the appropriate conditions. 
In \cref{s.model} we define the local model of our multi-parametric perturbations and perform some computations concerning those.
In \cref{s.central_proof} we use the results of the previous sections and a tower construction to prove \cref{p.central}.
In \cref{s.torus} we prove \cref{t.T3}.

\subsection{Acknowledgements} 

We thank the referees for several corrections and suggestions that improved the readability of the paper.

\section{Reduction to a central proposition} \label{s.reduction}

Let $H \coloneqq \{(\xi_1,\dots,\xi_d) \in \R^d \st {\textstyle \sum} \xi_j = 0\}$.
Define a linear isomorphism $T \colon H \to \R^{d-1}$ by:
\begin{equation}\label{e.def_T}
T \big( \xi_1,\xi_2, \dots, \xi_{d-1} , \xi_d \big) \coloneqq \big( \xi_1, \xi_1 + \xi_2, \dots, \xi_1 + \cdots + \xi_{d-1} \big).
\end{equation}

Given an ergodic $f \in \Diff_m^\reg(M)$, recall that $\boldsymbol{\lambda}(f)$ denotes the Lyapunov spectrum of $f$ with respect to the volume measure $m$.
We define $\hat{\boldsymbol{\lambda}}(f) \coloneqq T(\boldsymbol{\lambda}(f))$,
that is, $\hat{\boldsymbol{\lambda}}(f)$ is the vector whose $j$-th entry $\hat{\lambda}_j(f)$ is the sum of the $j$ biggest Lyapunov exponents.
This is a natural object to consider because $\hat{\lambda}_j(f)$ equals the top Lyapunov exponent of the $j$-fold exterior power of the derivative cocycle.
It is also convenient to define $\hat{\lambda}_0(f) \coloneqq 0 \eqcolon \hat{\lambda}_d(f)$.
The fact that $\lambda_1(f) \ge \cdots \ge \lambda_d(f)$ means that the function $j \in \{0,\dots,d\} \mapsto \hat{\lambda}_j(f)$ is concave: see \cref{f.concave}.

\begin{figure}[hbt]
	\begin{tikzpicture} 
		\draw [<->] (0,3)--(0,0)--(5.5,0);
		\coordinate (A0) at (0,0);
		\coordinate (A1) at (1,2);		\coordinate (B1) at (1,1.3);		
		\coordinate (A2) at (2,2.5);	\coordinate (B2) at (2,2.1);
		\coordinate (A3) at (3,2.25);	\coordinate (B3) at (3,1.9);
		\coordinate (A4) at (4,1.5);	\coordinate (B4) at (4,1.3);
		\coordinate (A5) at (5,0);
		\draw [-,thick] (A0)--(A1)--(A2)--(A3)--(A4)--(A5);
		\draw [-,thick] (A0)--(B1)--(B2)--(B3)--(B4)--(A5);
		\fill (A0) circle (1.5pt);
		\fill (A1) circle (1.5pt);
		\fill (A2) circle (1.5pt);
		\fill (A3) circle (1.5pt);
		\fill (A4) circle (1.5pt);
		\fill (A5) circle (1.5pt);
		\fill (B1) circle (1.5pt);
		\fill (B2) circle (1.5pt);
		\fill (B3) circle (1.5pt);
		\fill (B4) circle (1.5pt);
	\end{tikzpicture}
	\caption{On the top, the graph of $j \in \{0,\dots,d\} \mapsto \hat{\lambda}_j(f)$ for some $f$. The bottom graph corresponds to $T(\boldsymbol{\xi})$ for some ordered vector $\boldsymbol{\xi}$ satisfying assumptions \eqref{i.spec_12}--\eqref{i.spec_3} from \cref{t.majorize}.}
	\label{f.concave}
\end{figure}
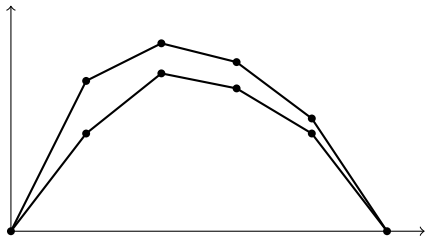

\medskip

In this \lcnamecref{s.reduction} we state \cref{p.central}, which roughly says that 
we can perturb $f$ in order to slightly lower the graph of $j \in \{0,\dots,d\} \mapsto \hat{\lambda}_j(f)$, and that different vertices of the graph can be moved somewhat independently.
We also show how \cref{p.central} implies the main theorem (\cref{t.majorize}).

\medskip

Given $u \in \{1,2,\dots,d-1\}$, 
define a ``gap function'' $\mathsf{g}_u \colon \R^d \to \R$ by:
\begin{equation}\label{e.def_gap}
\mathsf{g}_u(\xi_1,\dots,\xi_d) \coloneqq 
\min \big\{ \xi_1 - \xi_2, \xi_2 - \xi_3, \dots, \xi_{d-1} - \xi_d , \ \xi_u, \ -\xi_{u+1} \big\} \, ;
\end{equation}
Note that if $f$ is a conservative Anosov diffeomorphism of unstable index 
$u$ and admitting a simple dominated splitting then $\mathsf{g}_u(\boldsymbol{\lambda}(f))$ is strictly positive; indeed it is the minimal gap between the $d+1$ numbers
\[ 
\lambda_1(f) > \cdots > \lambda_u(f) > 0 > \lambda_{u+1}(f) > \dots >\lambda_d(f) \, .
\]

\begin{proposition}[Central proposition]\label{p.central}
Let $u \in \{1,2,\dots,d-1\}$, and let $a_1$, \dots, $a_{d-1}$, $\sigma$, and $\delta_0$ be positive numbers.
Then there exists $\delta \in (0,\delta_0)$ so that the following holds:

If $f \in \Diff_m^\reg(M)$ is a conservative Anosov diffeomorphism of unstable index $u$, admitting a simple dominated splitting, and such that $\mathsf{g}_u(\boldsymbol{\lambda}(f)) \ge \sigma$, 
then there exists a continuous map
\[
t \in [0,1]^{d-1} \mapsto f_t \in \Diff_m^\reg(M)
\]
where $f_{(0,\dots,0)} = f$
and for each $t = (t_1,\dots,t_{d-1}) \in [0,1]^{d-1}$, the conservative diffeomorphism $f_t$ is Anosov of unstable index $u$, admits a simple dominated splitting, and, for each $j \in \{1,\dots,d-1\}$,
\begin{gather}
\label{e.box_bounds}
\phantom{t_j = 1 \quad \Rightarrow\quad}
\hat{\lambda}_j(f) - 4\delta a_j           < \hat{\lambda}_j(f_t) < \hat{\lambda}_j(f) + \phantom{1}\delta a_j \, , \\
\label{e.box_top}
t_j = 0 \quad \Rightarrow\quad 
\hat{\lambda}_j(f) - \phantom{1}\delta a_j < \hat{\lambda}_j(f_t) < \hat{\lambda}_j(f) + \phantom{1}\delta a_j \, , \\
\label{e.box_bottom}
t_j = 1 \quad\Rightarrow\quad 
\hat{\lambda}_j(f) - 4\delta a_j           < \hat{\lambda}_j(f_t) < \hat{\lambda}_j(f) - 2\delta a_j \, .
\end{gather}
\end{proposition}

So the numbers $a_j$ allow us to move some of the summed exponents $\hat{\lambda}_j$ faster than others, but the overall movement is controlled by the scaling factor $\delta$. Let us emphasize that a main point of the proposition is that the dependence of $\delta$ on $f$ is only through $u$ and $\sigma$, and that is the key for the possible iteration of the proposition to obtain \cref{t.majorize}.

The following consequence is intuitively obvious (see \cref{f.shaky_cube}):

\begin{corollary}\label{c.cube_image}
In \cref{p.central}, 
the set $\Lambda \coloneqq \{\hat{\boldsymbol{\lambda}}(f_t) \st t \in [0,1]^{d-1} \}$ satisfies:
\begin{equation*} 
\prod_{j=1}^{d-1} \, [\hat{\lambda}_j(f) - 2\delta a_j, \,  \hat{\lambda}_j(f) - \delta a_j] \, 
\, \subseteq \,
\Lambda
\, \subseteq \, 
\prod_{j=1}^{d-1} \, [\hat{\lambda}_j(f) - 4\delta a_j, \,  \hat{\lambda}_j(f) + \delta a_j] \, .
\end{equation*}
\end{corollary}

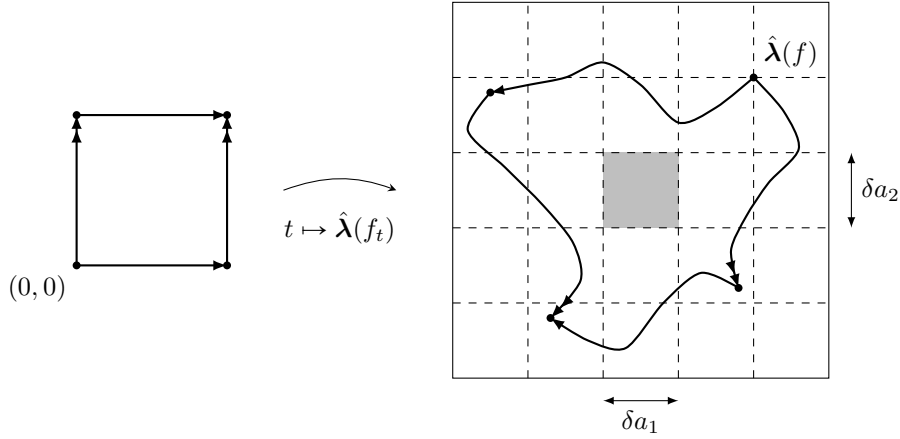
\begin{figure}[htb]
	\begin{center}
	\begin{tikzpicture}[font=\small, >=latex]

		\begin{scope}[shift={(-4,2.5)}]
			\draw [->,thick,line cap=round] (-1,-1)--(1,-1);
			\draw [->>,thick,line cap=round] (-1,-1)--(-1,1);
			\draw [->>,thick,line cap=round] (1,-1)--(1,1);
			\draw [->,thick,line cap=round] (-1,1)--(1,1);
			\fill (-1,-1) circle (1.5pt);
			\fill (1,-1) circle (1.5pt);
			\fill (-1,1) circle (1.5pt);
			\fill (1,1) circle (1.5pt);
			\node [below left] at (-1,-1) {$(0,0)$};
			\draw[-stealth] (1.75,0) to [out=20,in=160] (3.25,0);
			\node at (2.5,-.5) {$t \mapsto \hat{\boldsymbol{\lambda}}(f_t)$};
		\end{scope}
		
		\draw (0,0) rectangle (5,5);
		\fill[gray,opacity=.5] (2,2) rectangle (3,3);
				
		\draw [thin,dashed] (0,1)--(5,1);
		\draw [thin,dashed] (0,2)--(5,2);
		\draw [thin,dashed] (0,3)--(5,3);
		\draw [thin,dashed] (0,4)--(5,4);
		\draw [thin,dashed] (1,0)--(1,5);
		\draw [thin,dashed] (2,0)--(2,5);
		\draw [thin,dashed] (3,0)--(3,5);
		\draw [thin,dashed] (4,0)--(4,5);
				
		\coordinate (NW) at (4,4);
		\coordinate (NE) at (.5, 3.8);
		\coordinate (SW) at (3.8, 1.2);
		\coordinate (SE) at (1.3, .8);
						
		\draw [->,thick,line cap=round] plot [smooth] coordinates {(NW) (3.5,3.6) (3, 3.4) (2.5, 3.9) (2, 4.2) (1.5, 4) (1,3.9) (NE)};
	
		\draw [->>,thick,line cap=round] plot [smooth] coordinates {(NW) (4.4, 3.55) (4.6, 3.05) (4.2, 2.6) (3.9, 2.15) (3.7, 1.7) (SW)};
	
		\draw [->>,thick,line cap=round] plot [smooth] coordinates {(NE) (.2, 3.3) (.7, 2.8) (1.2, 2.3) (1.6, 1.8) (1.7, 1.3) (SE)};
	
		\draw [->,thick,line cap=round] plot [smooth] coordinates {(SW) (3.3,1.4) (2.8, 1) (2.3, .4) (1.8, .5) (SE)};

		\fill (NW) circle (1.5pt);
		\fill (NE) circle (1.5pt);
		\fill (SW) circle (1.5pt);
		\fill (SE) circle (1.5pt);
		\node [above right] at (NW) {$\hat{\boldsymbol{\lambda}}(f)$};

		\draw [<->] (2,-.3)--(3,-.3) node[below,midway]{$\delta a_1$};
		\draw [<->] (5.3,2)--(5.3,3) node[right,midway]{$\delta a_2$};
		
	\end{tikzpicture}
	\end{center}
	\caption{Illustration of \cref{p.central} for $d=3$. The images of the edges of the square $[0,1]^2$ under the map $t \mapsto \hat{\boldsymbol{\lambda}}(f_t)$ stay on the strips determined by conditions \eqref{e.box_top} and \eqref{e.box_bottom}. \cref{c.cube_image} tells us that the image of this map is a set $\Lambda$ that contains the small gray square and is contained in the big square.}
	\label{f.shaky_cube}
\end{figure}

\begin{proof}[Proof of \cref{c.cube_image}]
The second inclusion comes from \eqref{e.box_bounds}.
For the first one, we need the following topological 
fact:

\begin{lemma}\label{l.topological}
Let $\phi = (\phi_1,\dots,\phi_{d-1}) \colon [-1,1]^{d-1} \to \R^{d-1}$ be a continuous map such that for every $z = (z_1,\dots,z_{d-1}) \in [-1,1]^{d-1}$ and every $j \in \{1,\dots,d-1\}$ we have:
\[
\begin{gathered}
z_j = -1 \quad \Rightarrow \quad \phi_j(z) < -\tfrac{1}{3} \, , \\
z_j = \phantom{+}1 \quad \Rightarrow \quad \phi_j(z) > \phantom{+}\tfrac{1}{3} \, .
\end{gathered}
\]
Then the image of $\phi$ contains the cube $[-\tfrac{1}{3},\tfrac{1}{3}]^{d-1}$.
\end{lemma}

\begin{proof}
Let $C \coloneqq [-1,1]^{d-1}$.
Note that:
\begin{equation}\label{e.remark}
\text{$\forall$ $z \in \partial C$, the segment $[z,\phi(z)]$ does not intersect the cube $\tfrac{1}{3}C$.}
\end{equation}
Consider the map $\psi \colon C \to \R^{d-1}$ that coincides with a rescaled version of $\phi$ on the subcube $\tfrac{1}{2}C$, and on the remaining shell interpolates linearly between $\phi|_{\partial C}$ and $\id|_{\partial C}$.
More precisely, letting $\|\mathord{\cdot}\|_\infty$ denote the maximum norm in $\R^{d-1}$ (whose unit closed ball is the cube $C$), we define: 
\[
\psi(z)
\coloneqq 
\begin{cases}
\phi(2z) &\text{if $\|z\|_\infty \le 1/2$;}\\
 2(1-\|z\|_\infty) \phi \left( \|z\|_\infty^{-1} z\right) + (2-\|z\|_\infty^{-1}) z &\text{otherwise.}
\end{cases}
\]
In particular, $\psi$ coincides with the identity on the boundary $\partial C$.

By contradiction, suppose that $\tfrac{1}{3}C \setminus \phi(C)$ contains a point $w$.
Then it follows from observation \eqref{e.remark} that the image of $\psi$ does not contain $w$.
Fix a retraction $\pi$ of $\R^{d-1}\setminus \{w\}$ onto $\partial C$.
Then $\pi \circ \psi$ is a retraction of the cube $C$ onto its boundary.
It is a known fact that no such map exists.
This contradiction completes the proof of the \lcnamecref{l.topological}.
\end{proof}

Now \cref{c.cube_image} follows by an affine change of coordinates.
Namely, consider the map: 
\begin{equation*}
\phi(z) \coloneqq B \left( \hat{\boldsymbol{\lambda}}\left(f_{A(z)}\right) \right) 
\end{equation*}
where $A = (A_1,\dots,A_{d-1})$, $B=(B_1,\dots,B_{d-1})$ are the maps:
\[
A_j(z_1,\dots,z_{d-1})   \coloneqq \frac{1 - z_j}{2} \, , \quad
B_j(\xi_1,\dots,\xi_{d-1})\coloneqq 1 + \frac{2}{3} \, \frac{\xi_j-\hat{\lambda}_j(f)}{\delta a_j} \, .
\]
By \eqref{e.box_top}, if $z_j=1$, then $\tfrac{1}{3}<\phi_j(z)<\tfrac{5}{3}$, while
by \eqref{e.box_bottom}, if $z_j=-1$, then $-\tfrac{5}{3}<\phi_j(z)<-\tfrac{1}{3}$.
So we can apply \cref{l.topological} and conclude the proof of \cref{c.cube_image}.
\end{proof}

\begin{proof}[Proof of \cref{t.majorize}]
Let $f$ and $\boldsymbol{\xi}$ be as in the statement of the \lcnamecref{t.majorize}.
Let 
\[
\sigma \coloneqq \tfrac{1}{2} \min \big\{ \mathsf{g}_u(\boldsymbol{\lambda}(f)) , \mathsf{g}_u(\boldsymbol{\xi}) \big\} \, . 
\]
Recalling \eqref{e.def_T}, let $\hat{\boldsymbol{\xi}} \coloneqq T(\boldsymbol{\xi})$ and
$(a_1, \dots, a_{d-1}) \coloneqq \hat{\boldsymbol{\lambda}}(f) - \hat{\boldsymbol{\xi}}$.
Since $\boldsymbol{\lambda}(f)$ strictly majorizes $\boldsymbol{\xi}$, each $a_j$ is positive.
The function $\mathsf{g}_u \circ T^{-1} \colon \R^{d-1} \to \R$ is concave (since it is the minimum of affine functions), and it follows that this function is $\ge 2\sigma$ on the segment 
$[ \hat{\boldsymbol{\xi}} , \hat{\boldsymbol{\lambda}}(f)]$.
By continuity, we can find a small positive $\delta_0 < 1$ such that 
the function $\mathsf{g}_u \circ T^{-1}$ is $\ge \sigma$ on the following neighborhood of the segment (pictured in \cref{f.proof_thrm}):
\[
V \coloneqq [ \hat{\boldsymbol{\xi}} , \hat{\boldsymbol{\lambda}}(f)] 
+ \prod_{j=1}^{d-1} [ - 4\delta_0 a_j, \delta_0 a_j] \, .
\]

Let $\delta = \delta (a_1, \dots, a_{d-1}, \sigma, \delta_0)$ be given by \cref{p.central}.
Let $n \coloneqq \lfloor \delta^{-1} \rfloor$; this is a positive integer since $0 < \delta < \delta_0 < 1$.
For each $i \in \{0,1,\dots,n\}$ let
\[
\boldsymbol{\eta}_i \coloneqq (1-\tfrac{i}{n})\hat{\boldsymbol{\lambda}}(f) + \tfrac{i}{n}\hat{\boldsymbol{\xi}} \, .
\]
Note that $\tfrac{1}{n} \in [\delta, 2 \delta]$.
Therefore, for each $i \in \{0,\dots,n-1\}$,
\begin{equation}\label{e.smart_step}
\boldsymbol{\eta}_{i+1} - \boldsymbol{\eta}_i \in \prod_{j=1}^{d-1}  [- 2\delta a_j, - \delta a_j] \, .
\end{equation}

We now construct a continuous path $(f_s)_{s \in [0,1]}$ of conservative Anosov diffeomorphisms as follows.
Applying \cref{p.central} to the diffeomorphism $f_0 \coloneqq f$,
we obtain a certain family of Anosov diffeomorphisms $(g_{0,t})$, where $t$ runs in the cube $[0,1]^{d-1}$.
Since $\hat{\boldsymbol{\lambda}}(f_0) = \boldsymbol{\eta}_0$, by \cref{c.cube_image} and \eqref{e.smart_step}, there exists $t_0$ in the cube such that 
$\hat{\boldsymbol{\lambda}}(g_{0,t_0}) = \boldsymbol{\eta}_1$.
Let $f_{1/n} \coloneqq g_{0,t_0}$, 
and define $f_s$ for $s$ in the interval $[0,1/n]$ by 
$f_s \coloneqq g_{0, n t_0 s}$.
Note that these diffeomorphisms obey the gap condition $\mathsf{g}_u(\boldsymbol{\lambda}(f_s)) \ge \sigma$.
We continue recursively in the obvious way:
we apply \cref{p.central,c.cube_image} to the diffeomorphism $f_{1/n}$, extend the family $f_t$ to the interval $[1/n,2/n]$ and so on.
We end up defining a path $(f_s)_{s \in [0,1]}$ 
of conservative Anosov diffeomorphisms of unstable index $u$ admitting simple dominated splittings, 
such that $\hat{\boldsymbol{\lambda}}(f_{i/n}) = \boldsymbol{\eta}_i$ for each $i \in \{0,\dots, n\}$.
In particular, $\hat{\boldsymbol{\lambda}}(f_1) = \hat{\boldsymbol{\xi}}$,
or equivalently, $\boldsymbol{\lambda}(f_1) = \boldsymbol{\xi}$,
as desired.
\end{proof}

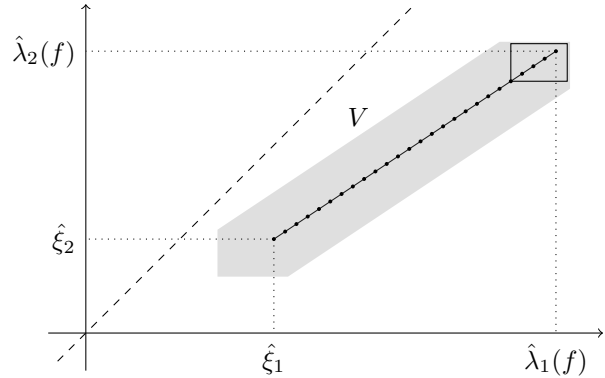
\begin{figure}[htb]
	\begin{center}
	\begin{tikzpicture}[font=\small,scale=2.5]
			\draw[->] (-.2,0) -- (2.75,0); 
			\draw[->] (0,-.2) -- (0,1.75); 
			\draw[dashed] (-.15,-.15) -- (1.75,1.75); 
			\draw (1,.5) -- (2.5,1.5); 
			
			\fill[gray,opacity=.25] (.7,.3)--(1.075,.3)--(2.575,1.3)--(2.575,1.55)--(2.2,1.55)--(.7,.55)--cycle;
			\node at (1.45,1.05)[above]{$V$};


			\foreach \i in {0,1,...,25}{
				\fill (1+.06*\i,.5+.04*\i) circle (.3pt);
			}
			\draw (2.26,1.34) rectangle (2.56,1.54);
			
			\draw[dotted](1,.5)--(1,0)		node[below]{$\hat{\xi}_1$};
			\draw[dotted](1,.5)--(0,.5)		node[left]{$\hat{\xi}_2$};
			\draw[dotted](2.5,1.5)--(2.5,0)	node[below]{$\hat{\lambda}_1(f)$};
			\draw[dotted](2.5,1.5)--(0,1.5)	node[left]{$\hat{\lambda}_2(f)$};
		
	\end{tikzpicture}
	\end{center}
	\caption{Illustration of the proof of \cref{t.majorize} with $d=3$, $u=1$. The function $\mathsf{g}_u \circ T^{-1}$ is positive on the sector between the horizontal positive semi-axis and the diagonal. The gray region is the neighborhood $V$, and the marked points along the segment $[ \hat{\boldsymbol{\xi}} , \hat{\boldsymbol{\lambda}}(f)]$ are the $\boldsymbol{\eta}_i$'s. For the first perturbation $(g_{0,t})$, the corresponding hatted Lyapunov vector $\hat{\boldsymbol{\lambda}}(g_{0,t})$ stays inside the upper right rectangle, and hits $\boldsymbol{\eta}_1$ for some parameter $t=t_0$.}
	\label{f.proof_thrm}
\end{figure}

\begin{remark}[Prescribing paths of spectra]
The proof of \cref{t.majorize} can be easily adapted so that the path  $(\boldsymbol{\lambda}(f_t))_{t\in [0,1]}$ is $C^0$-close to any given path from $\boldsymbol{\lambda}(f)$ to $\boldsymbol{\xi}$ that satisfies property \eqref{i.spec_12}  and is monotone with respect to the majorization partial order.
\end{remark}

\medskip

In \cref{s.adapted,s.damping,s.model} we establish several preliminary results which will be eventually used to prove the central proposition (\cref{p.central}) in \cref{s.central_proof}.

\section{Lyapunov metrics and charts} \label{s.adapted}

\subsection{Lyapunov metrics}

There are different ways to introduce Riemannian metrics that are suitable to study particular classes of hyperbolic dynamical systems. Such metrics are usually called Lyapunov or adapted metrics, and both terms are used in more than one meaning. For various versions of such Lyapunov metrics, see e.g.\ \cite{BP, Gourmelon, KH}.
Here we will construct a variant which specifically fits our setting.

Suppose that $f \in \Diff_m^\reg(M)$ is a conservative Anosov diffeomorphism admitting a simple dominated splitting $TM = E_1 \oplus \cdots \oplus E_d$. Given a $C^0$ Riemannian metric $\biangle{\mathord{\cdot}, \mathord{\cdot}}$, 
let $\tribar{\mathord{\cdot}}$ denote the induced vector norm, 
and consider the \emph{expansion functions} $\chi_1$, \dots, $\chi_d \colon M \to \R$ defined by:
\begin{equation}\label{e.def_chi}
\chi_j(x) \coloneqq \log  \frac{\tribar{Df(x) v}}{\tribar{v}} , \quad
\text{for arbitrary nonzero } v\in E_j(x) \, .
\end{equation}
It is clear that each $\chi_j$ is continuous and its integral is $\lambda_j(f)$.
We say that $\biangle{\mathord{\cdot}, \mathord{\cdot}}$ is a \emph{Lyapunov metric} if:
\begin{itemize}
\item the bundles $E_j$ are mutually orthogonal with respect to  $\biangle{\mathord{\cdot}, \mathord{\cdot}}$;
\item the expansion functions satisfy, for every $x \in M$,
\begin{equation}\label{e.adapted_pointwise}
\chi_1(x) > \chi_2(x) > \cdots > \chi_u(x) > 0 > \chi_{u+1}(x) > \cdots > \chi_d(x) \, ,
\end{equation}
where $u$ is the unstable index of $f$.
\end{itemize}

\begin{proposition}[Lyapunov metric with $L^1$ estimate] \label{p.adapted_metric}
Let $f \in \Diff_m^\reg(M)$ be a conservative Anosov diffeomorphism admitting a simple dominated splitting. 
Furthermore, let $\epsilon>0$.
Then there exists a Lyapunov metric such that each expansion function is $L^1$-close to a constant:
\begin{equation}\label{e.adapted_L1}
\int_M \big| \chi_j(x) - \lambda_j(f) | \dd m(x) < \epsilon \, .
\end{equation}
\end{proposition}


We could smoothen the metrics given by the \lcnamecref{p.adapted_metric}, but in that case we would lose the orthogonality of the bundles of the dominated splitting, which is convenient for our later calculations.

\begin{proof}[Proof of \cref{p.adapted_metric}]
We fix an arbitrary Riemannian metric on $M$.
Consider the following continuous functions:
\[
\theta_j^{(n)} (x) \coloneqq 
\log \frac{\| Df^n(x) v \|}{\|v\|} 
\quad \text{for arbitrary nonzero } v\in E_j(x).
\]
For each $j$, the sequence above forms an additive cocycle with respect to the dynamics $f$.

By domination and hyperbolicity, if $N$ is sufficiently large then:
\begin{equation}\label{e.pre_pointwise}
\forall x \in M, \ 
\theta_{1}^{(N)}(x) > \cdots > \theta_u^{(N)}(x) > 0 > \theta_{u+1}^{(N)}(x) > \cdots > \theta_{d}^{(N)}(x).
\end{equation}
On the other hand, the functions $\theta_j^{(n)} / n$ 
converge $m$-almost everywhere to the constant $\lambda_j(f)$.
So, increasing $N$ if necessary, we assume that:
\begin{equation}\label{e.pre_L1}
\int_M \bigg| \frac{\theta_j^{(N)}(x)}{N} - \lambda_j(f) \bigg| \dd m(x) < \epsilon \, .
\end{equation}

Now, for each $x \in M$, $j \in \{1,\dots,d\}$, and $v \in E_j(x)$, we let
\begin{equation}\label{e.magic}
\tribar{v} \coloneqq 
\prod_{n=0}^{N-1} \| Df^n(x) v \|^{1/N} \, .
\end{equation}
This defines a norm on each one-dimensional bundle $E_j(x)$.
Consider the unique inner product on $T_x M$ that makes those bundles orthogonal and whose induced norm agrees with the definition above.

If $v\in E_j(x)$ is nonzero then, by telescopic multiplication,
\[
\frac{\tribar{Df(x) v}}{\tribar{v}} = 
\frac{\| Df^N(x) v \|^{1/N}}{\|v\|^{1/N}} \, . 
\]
This means that the expansion function $\chi_j$ defined by \eqref{e.def_chi} equals $\theta_j^{(N)}/N$.
Therefore, the desired properties \eqref{e.adapted_pointwise} and \eqref{e.adapted_L1} 
are just \eqref{e.pre_pointwise} and \eqref{e.pre_L1}.
\end{proof}

\begin{remark}
It follows from the proof that the Lyapunov metrics can be constructed so that 
the tempering property $|\chi_j \circ f - \chi_j| < \epsilon$ also holds.
\end{remark}

\begin{remark} 
Our proof used that the bundles of the dominated splitting are one-dimensional, so that the geometric average defined by the formula \eqref{e.magic} defines a norm.
If the bundles $E_j$ are higher dimensional, one can still adapt this trick by taking an appropriate notion of averaging on the symmetric space of inner products: see \cite[p.~1839]{Bochi_ICM}.
\end{remark}

\subsection{Lyapunov charts}\label{ss.adapted_charts}

\begin{proposition}[Lyapunov charts] \label{p.adapted_charts}
Suppose that $f\in \Diff^\reg_m(M)$ is a conservative Anosov diffeomorphism of unstable index $u$ admitting a simple dominated splitting $TM = E_1 \oplus \cdots \oplus E_d$,
and that $\biangle{\mathord{\cdot}, \mathord{\cdot}}$ is a Lyapunov metric.
	
Then for all $x \in M$, there exists a map
$\Phi_x \colon B_0 \to M$, where $B_0 \subset \R^d$ is a fixed closed ball centered at $0$, with the following properties:
\begin{enumerate}
\item \label{i.charts_1}
$\Phi_x(0) = x$;
\item \label{i.charts_2}
$\Phi_x$ is a smooth diffeomorphism onto its image; 
\item \label{i.charts_3}
$\Phi_x$ has constant Jacobian, that is, the push-forward of Lebesgue measure on $B_0$ equals the restriction of the measure $m$ to $\Phi_x(B_0)$ times a constant factor;
\item \label{i.charts_4}
the derivative $\cL_x \coloneqq D\Phi_x(0)$
takes the canonical basis $\{e_1,\dots,e_d\}$ of $\R^d$
to a basis $\{\cL_x(e_1), \dots, \cL_x(e_d)\}$ of $T_x M$ which is orthonormal for the Lyapunov metric $\biangle{\mathord{\cdot}, \mathord{\cdot}}_x$,
and, moreover, $\cL_x(e_j) \in E_j(x)$ for each~$j$.
\end{enumerate}
Furthermore, $\{\Phi_x \st x \in M\}$ is a relatively compact subset of $C^\infty(B_0,M)$. 
\end{proposition}

\begin{proof}
As it is well-known and elementary (see e.g.\ \cite[p.~6]{Koba}),
the manifold $M$ admits a conservative atlas, that is, an atlas whose charts $F_i \colon U_i \subset \R^d \to M$ are such that the push-forward of Lebesgue measure on $U_i$ coincides with the measure $m$ restricted to $F_i(U_i)$. By compactness, we assume that this atlas is finite.

Now, given $x \in M$, let $\cL_x \colon \R^d \to T_x M$ be a linear isomorphism that sends the coordinate axes to the bundles $E_{1}$, \dots, $E_{d}$, and such that the push-forward of the standard inner product is the Lyapunov metric $\biangle{\mathord{\cdot}, \mathord{\cdot}}_x$. 
Choose $i$ such that $F_i(U_i) \ni x$, and define 
\[
\Phi_x (z) \coloneqq F_i \left( F_i^{-1}(x) +  (DF_i^{-1}(x) \circ \cL_x) (z) \right) \, ,
\]
for every $z$ in a sufficiently small closed ball $B_0 \subset \R^d$ around $0$.
It is clear that these maps have the asserted properties \eqref{i.charts_1}--\eqref{i.charts_4}.
Moreover, all their derivatives are uniformly bounded, so we obtain relative compactness of the family.
\end{proof}

The maps $\Phi_x$ given by the \lcnamecref{p.adapted_charts} are called \emph{Lyapunov charts}.
Note that the derivative of $f$ at an arbitrary point $x\in M$ can be diagonalized using the Lyapunov charts as follows:
\begin{equation}\label{e.diagonal}
D ( \Phi_{f(x)}^{-1} \circ f \circ \Phi_x) (0)  =
\begin{pmatrix}
\pm e^{\chi_{1}(x)}	&			& 0 					\\
	 				& \ddots	&  						\\
0					& 			& \pm e^{\chi_{d}(x)}
\end{pmatrix}
\, .
\end{equation}

\section{Damping perturbations} \label{s.damping}

In this \lcnamecref{s.damping} we define a certain type of perturbations $\tilde f$ of an Anosov diffeomorphism $f$, called \emph{damping} perturbations. The idea comes from Baraviera and Bonatti \cite{BarBon} but, as one of the referees has pointed out to us, constructions of this type go back to Newhouse and Ma\~n\'e.

Damping perturbations are actually ``large perturbations'', in the sense that the $C^1$ or even the $C^0$ distance between $\tilde f$ and $f$ may be large. On the other hand, the support of the perturbation, i.e., the set $Z \coloneqq \{x \st \tilde f(x) \neq f(x)\}$ is required to be ``dynamically small'': the return times from $Z$ to itself are not smaller than a large number $N$. On $Z$ we impose a transversality condition: essentially we want the spaces $D\tilde{f}(x)E^\mathrm{u}(x)$ and $D\tilde{f}(x)E^\mathrm{s}(x)$ to be transverse to $E^\mathrm{s}(\tilde{f}(x))$ and $E^\mathrm{u}(\tilde{f}(x))$, respectively, where $E^\mathrm{u}\oplus E^\mathrm{s}$ denotes the hyperbolic splitting of $f$.
We show that if the least return time $N$ is large enough then the perturbed diffeomorphism $\tilde{f}$ is still Anosov. Moreover, on the set $Z$ the new unstable bundle $\tilde{E}^\mathrm{u}$ is close to the old one $E^\mathrm{u}$, while on the set $f(Z)$ they can be far apart. As we go upwards in the tower with base $Z$, the bundle $\tilde{E}^\mathrm{u}$ is attracted by $E^\mathrm{u}$, so it suffers some ``damping'' before returning to $Z$ and getting ``kicked'' again. See \cref{f.damping}. Actually we work with Anosov diffeomorphisms with simple dominated splitting; we show that domination also persists under damping perturbations. Let us give precise statements.

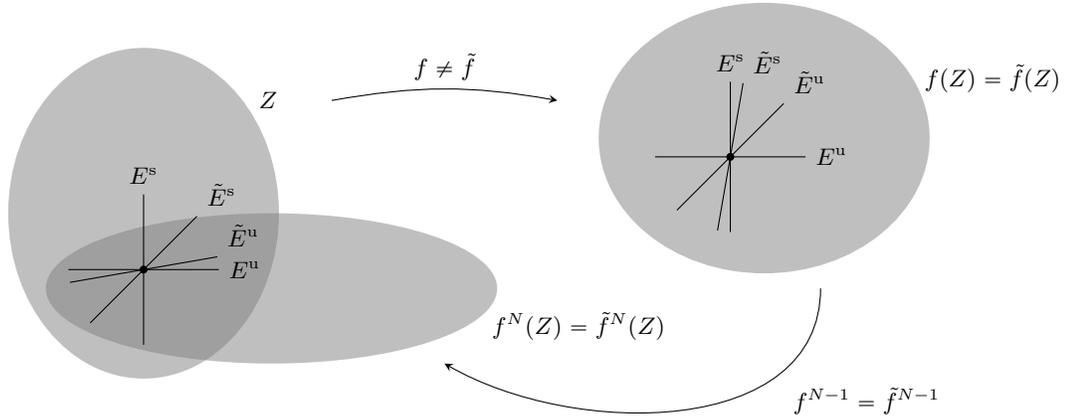
\begin{figure}[hbt]
	\begin{tikzpicture}[font=\footnotesize]
		
		\fill[gray,opacity=.5] (1.7,-1) ellipse [x radius = 3, y radius = 1];
		\node[anchor=west] at (4.5,-1.5) {$f^N(Z) = \tilde{f}^{N}(Z)$};
			
		\draw[-stealth] (9,-1) [out=270, in=330] to (4,-2);
		\node[anchor=west] at (8.5,-2.5) {$f^{N-1} = \tilde{f}^{N-1}$};

		\fill[gray,opacity=.5] (0,0) ellipse [x radius = 1.8, y radius = 2.2];
		\node[anchor=west] at (1.4,1.5) {$Z$};
		\begin{scope}[shift={(0,-.75)}]
		\draw (-1,0)--(1,0) node[right]{$E^\mathrm{u}$};
		\draw (0,-1)--(0,1) node[above]{$E^\mathrm{s}$};
		\draw (-.98,-.17)--(.98,.17) node[above right]{$\tilde{E}^\mathrm{u}$};
		\draw (-.71,-.71)--(.71,.71) node[above right]{$\tilde{E}^\mathrm{s}$};
		\fill (0,0) circle (1.5pt);
		\end{scope}

		\draw[-stealth] (2.5,1.5) to [out=10,in=170] (5.5,1.5);
		\node at (4,1.95) {$f \neq \tilde f$};
		
		\begin{scope}[shift={(8.25,1)}]
		\fill[gray,opacity=.5] (0,0) ellipse [x radius = 2.2, y radius = 1.8];	
		\node[anchor=west] at (2,.8) {$f(Z) = \tilde{f}(Z)$};
		\begin{scope}[shift={(-.45,-.25)}]
		\draw (-1,0)--(1,0) node[right]{$E^\mathrm{u}$};
		\draw (0,-1)--(0,1) node[above]{$E^\mathrm{s}$};
		\draw (-.71,-.71)--(.71,.71) node[above right]{$\tilde{E}^\mathrm{u}$};
		\draw (-.17,-.98)--(.17,.98) node[above right]{$\tilde{E}^\mathrm{s}$};
		\fill (0,0) circle (1.5pt);
		\end{scope}
		\end{scope}
		
	\end{tikzpicture}
	\caption{A damping perturbation $\tilde f$ of an Anosov diffeomorphism $f$.}
	\label{f.damping}
\end{figure}

\subsection{Definition of damping perturbations}

Assume given a conservative Anosov diffeomorphism $f\in \Diff^\reg_m(M)$ of unstable index $u$ and admitting a simple dominated splitting $TM = E_1 \oplus \cdots \oplus E_d$.
Fix a Lyapunov metric $\biangle{\mathord{\cdot}, \mathord{\cdot}}$.
Let us use the following notation:
\begin{equation}\label{e.vec_chi}
\boldsymbol{\chi}(x) \coloneqq (\chi_1(x), \dots, \chi_d(x)) \, ,
\end{equation}
where the $\chi_j$'s are the expansion functions \eqref{e.def_chi}.
By definition of Lyapunov metric, 
$\mathsf{g}_u \big( \boldsymbol{\chi}(x) \big) > 0$ for every $x \in M$, where $\mathsf{g}_u$ is the gap function \eqref{e.def_gap}.

For each $x\in M$, let
\begin{equation}\label{e.L}
\cL_x \colon \R^d \to T_x M
\end{equation}
be a linear map that takes the canonical basis $\{e_1,\dots,e_d\}$ of $\R^d$
to a basis $\{\cL_x(e_1), \dots, \cL_x(e_d)\}$ of $T_x M$ that is orthonormal for the Lyapunov metric $\biangle{\mathord{\cdot}, \mathord{\cdot}}_x$,
and, moreover, $\cL_x(e_j) \in E_j(x)$ for each $j$.
Note that the map $\cL_x$ is unique modulo composing from the right by a matrix of the form $\diag(\pm 1, \dots, \pm 1)$.
Taking quotient by this finite group, $\cL_x$ becomes unique and continuous.

For each $j \in \{1,\dots,d-1\}$ and $\tau>0$, 
we define the 
\emph{standard horizontal and vertical cones of index $j$ and opening $\tau$} as
the following subsets of Euclidean space $\R^d$:
\begin{align}
\label{e.hor_cone}
\cH_j(\tau) &\coloneqq 
\left\{ (z_1,\dots,z_d) \st 
z_{j+1}^2 + \dots + z_d^2 < \tau^2 (z_1^2 + \dots + z_j^2) \right\} \cup \{0\} \, , \\
\label{e.ver_cone}
\cV_j(\tau) &\coloneqq 
\left\{ (z_1,\dots,z_d) \st 
z_1^2 + \dots + z_j^2 < \tau^2 (z_{j+1}^2 + \dots + z_d^2) \right\}  \cup \{0\} \, .
\end{align}
Then we define the following continuous fields of cones on the tangent bundle $TM$:
\[
\cH_j(x,\tau) \coloneqq \cL_x (\cH_j(\tau)) \, , \quad 
\cV_j(x,\tau) \coloneqq \cL_x (\cV_j(\tau)) \, .
\]
By \eqref{e.diagonal}, we have the following invariance properties, for every $\tau>0$,:
\begin{align}
\label{e.hor_cone_field_inv}
Df(x)     \cH_j(x,\tau) &\subseteq \cH_j\big(f(x),     e^{-[\chi_j(x)-\chi_{j+1}(x)]} \tau\big) \, , \\
\label{e.ver_cone_field_inv}
Df^{-1}(x)\cV_j(x,\tau) &\subseteq \cV_j\big(f^{-1}(x),e^{-[\chi_j(f^{-1}(x))-\chi_{j+1}(f^{-1}(x))]} \tau\big) \, .
\end{align}

Let $P^\mathrm{u}$ and $P^\mathrm{s} : TM \to TM$ be the projections on the unstable and stable bundles of $f$, respectively, so that their sum is the identity.

\medskip

Fix numbers $\alpha > \beta >0$, $\kappa>0$, $\sigma > 0$, and $N \ge 2$.
We say that a diffeomorphism $\tilde f \in \Diff^\reg_m(M)$ is a \emph{$(\alpha,\beta,\kappa,\sigma,N)$-damping perturbation} of $f$ with respect to the metric $\biangle{\mathord{\cdot}, \mathord{\cdot}}$ if the following conditions hold:
\begin{enumerate}[label=(\roman*),ref=\roman*]

\item\label{i.damping_cones} 
for all $x \in M$ and $j \in \{1,\dots,d-1\}$, we have:
\begin{align*}
D\tilde{f}(x)     \cH_j(x,\beta) &\subseteq \cH_j(\tilde{f}(x),     \alpha) \, , \\
D\tilde{f}^{-1}(x)\cV_j(x,\beta) &\subseteq \cV_j(\tilde{f}^{-1}(x),\alpha) \, .
\end{align*} 

\item\label{i.damping_norm} 
For all $x \in M$, we have
\begin{alignat*}{3}
v &\in \cH_u(x,\beta) &\quad &\Rightarrow &\quad \tribar{P^\mathrm{u} \, D\tilde{f}(x) \, v} &\ge \kappa \, \tribar{P^\mathrm{u} \, v} \, , \\
v &\in \cV_u(x,\beta) &\quad &\Rightarrow &\quad \tribar{P^\mathrm{s} \, D\tilde{f}^{-1}(x) \, v} &\ge \kappa \, \tribar{P^\mathrm{s} \, v} \, .
\end{alignat*}

\item\label{i.damping_tower} 
Letting 
\[
Z \coloneqq \{x \in M \st \tilde{f}(x) \neq f(x)\;\mbox{or}\; D\tilde{f}(x)\neq Df(x)\} \, ,
\] 
the sets 
$Z$, $f(Z)$, \dots, $f^{N-1}(Z)$ have disjoint closures.

\item\label{i.damping_gap} 
For all $x \in Z$ we have, in terms of notations \eqref{e.vec_chi} and \eqref{e.def_gap}:
\[
\mathsf{g}_u \left( \frac{1}{N-1} \sum_{n=1}^{N-1} \boldsymbol{\chi}(f^n x) \right) \ge \frac{\sigma}{2} \, ,
\quad
\mathsf{g}_u \left( \frac{1}{N-1} \sum_{n=-N-2}^{0} \boldsymbol{\chi}(f^n x) \right) \ge \frac{\sigma}{2} \, .
\]
\end{enumerate}

Note that the definition is symmetric under time-reversal, i.e.\ $\tilde{f}^{-1}$ is a $(\alpha,\beta,\kappa,\sigma,N)$-damping perturbation of $f^{-1}$.
Indeed, $\{x \in M \st \tilde{f}^{-1}(x) \neq f^{-1}(x)\;\mbox{or}\;D\tilde{f}^{-1}(x) \neq Df^{-1}(x)\} = f(Z)$.

\begin{remark}\label{e.no_uniformity_needed}
Note that condition \eqref{i.damping_gap} is weaker than the condition $\mathsf{g}_u \circ \boldsymbol{\chi} \ge \frac{\sigma}{2}$, which is always satisfied in this paper.
The reason why we insist in working with the more complicated condition \eqref{i.damping_gap} is that it may be useful for future applications of our methods.
Moreover, working with the stronger condition would not make the proof of \cref{p.damping} below any simpler.
\end{remark}

\subsection{Properties of damping perturbations}\label{ss.damping_properties}

We will show that damping perturbations are still Anosov with simple dominated splitting, provided the ``damping time'' $N$ is large enough. It is important that the condition on $N$ depends only on the other parameters $\alpha$, $\beta$, $\kappa$, $\sigma$ (together with a parameter $\gamma$ that controls the desired closeness between the new and the old spaces), but not on $f$ itself. Furthermore, given the other parameters, the damping time $N$ must be at least of the order $1/\sigma$. That should not come as a surprise, since Lyapunov exponents are inversely proportional to the time it takes to obtain a prescribed factor of expansion or contraction.

\medskip

Given numbers $\alpha>\beta>\gamma>0$, $\kappa>0$, and $\sigma>0$, let 
\begin{equation}\label{e.N0}
N_0 \coloneqq \left\lfloor \frac{2}{\sigma}\log\left(\max\{\kappa^{-1},\alpha\gamma^{-1}\}\right)\right\rfloor + 2 \, .
\end{equation}
Let $f\in \Diff^\reg_m(M)$ be a conservative Anosov diffeomorphism of unstable index $u$ admitting a simple dominated splitting $TM = E_1 \oplus \cdots \oplus E_d$,
and fix a Lyapunov metric for $f$.

\begin{proposition}\label{p.damping}
Let $N \ge N_0$ and let $\tilde f \in \Diff^\reg_m(M)$ be a $(\alpha,\beta,\kappa,\sigma,N)$-damping perturbation of $f$ with respect to the Lyapunov metric.

Then $\tilde f$ is a conservative Anosov diffeomorphism of unstable index $u$, and it admits a simple dominated splitting $\tilde E_1 \oplus \cdots \oplus \tilde E_d$.

Moreover, letting $Z \coloneqq \{x \in M \st \tilde{f}(x) \neq f(x)\;\mbox{or}\; D\tilde{f}(x)\neq Df(x)\}$, then for all $x\in M$ and $j \in \{1,\dots,d-1\}$ we have: 
\begin{enumerate}
\item\label{i.hor_perturb} 
$\tilde{E}_1(x) \oplus \dots \oplus \tilde{E}_j(x)$ and $E_{j+1}(x) \oplus \dots \oplus E_d(x)$ are transverse; and\\ if $x \not\in \bigsqcup_{n=1}^{N-1} f^n(Z)$, then $\tilde{E}_1(x) \oplus \dots \oplus \tilde{E}_j(x)\subset \cH_j(x,\gamma)$;

\item\label{i.ver_perturb} 
$\tilde{E}_{j+1}(x) \oplus \dots \oplus \tilde{E}_d(x)$ and $E_1(x) \oplus \dots \oplus E_j(x)$ are transverse; and\\ if $x \not\in \bigsqcup_{n=-N-2}^{0} f^n(Z)$, then $\tilde{E}_{j+1}(x) \oplus \dots \oplus \tilde{E}_d(x)\subset \cV_j(x,\gamma)$.

\end{enumerate}
\end{proposition}


\begin{proof}[Proof of \cref{p.damping}]
Observe that  $N_0 \ge 2$ and that
\begin{equation}\label{e.def_N}
e^{(N_0-1) \sigma/2} > \max\left\{\kappa^{-1}  , \alpha \gamma^{-1} \right\} \, .
\end{equation}

The first step is to construct continuous fields of cones that are invariant with respect to the perturbation $\tilde f$:

\begin{lemma}\label{l.new_cone}
For each $j \in \{1,\dots,d-1\}$,
there exists a continuous function $\omega_j$ on $M$ such that, for every $x \in M$,
\begin{enumerate}[itemsep=1ex]
\item \label{i.omega_bounds}
$\gamma \le \omega_j(x) \le \alpha$;

\item \label{i.omega_flat}
$\omega_j(x) = \gamma$ if $x \not\in \bigsqcup_{n=1}^{N-1} f^n(\bar Z)$;

\item \label{i.forward_invariance}
$D\tilde f(x) \, \overline{\cH_j(x,\omega_j(x))} \subset \cH_j \big( \tilde f(x), \omega_j(\tilde f(x)) \big)$.

\end{enumerate}
\end{lemma}

\begin{proof}
Fix $j$.
By \eqref{e.def_N}, we can choose $\epsilon>0$ such that:
\[
\epsilon < \inf_M (\chi_j - \chi_{j+1})
\quad \text{and} \quad
\alpha e^{-(N-1)(\sigma/2-\epsilon)}<\gamma \, .
\]
Let
\[
A \coloneqq \left\{ x \in M \st
D\tilde{f}(x) \, \overline{\cH_j(x,\gamma)} \subset \cH_j(\tilde{f}(x),\gamma) \right\} \, .
\]
Then $A$ is an open set and $A \cup Z = M$.
Let $\rho_1 + \rho_2 = 1$ be a partition of unity subordinated to this open covering.
Define the function $\omega_j$ as follows:
\begin{itemize}[itemsep=1ex]
\item if $x \not\in \bigsqcup_{n=1}^{N-1} f^n(\bar Z)$ then we let $\omega_j(x) \coloneqq \gamma$;
\item if $x \in f(Z) = \tilde{f}(Z)$ then we let
\[
\omega_j(x) \coloneqq \gamma \rho_1(\tilde{f}^{-1}(x)) + \alpha \rho_2(\tilde{f}^{-1}(x)) \, .
\]
\item if $x \in f^n(Z)$ where $2 \le n \le N-1$ then, assuming $\omega_j$ was already defined on $f^{n-1}(Z)$, we let
\begin{equation*}  
\omega_j(x) \coloneqq \max \left\{ \gamma, \omega_j(f^{-1}(x)) e^{-\chi_j(f^{-1}(x)) +\chi_{j+1}(f^{-1}(x)) + \epsilon} \right\} \, .
\end{equation*}
\end{itemize}
The function $\omega_j$ is continuous and has properties \eqref{i.omega_bounds} and \eqref{i.omega_flat}.
Next we check property \eqref{i.forward_invariance} case by case.
First consider $x \in Z$.
\begin{itemize}[itemsep=1ex]

\item
If $x \in Z \cap A$ then $\omega_j(x) = \gamma$ and $\omega_j(\tilde{f}(x)) \ge \gamma$, so \eqref{i.forward_invariance} follows directly from the definition of $A$.

\item 
If $x \in Z \setminus A$ then $\omega_j(x) = \gamma < \beta$ and $\omega_j(\tilde{f}(x)) = \alpha$, so \eqref{i.forward_invariance} follows from condition \eqref{i.damping_cones} of the definition of damping perturbations.
\end{itemize}

Next we consider the cases where $x \not\in Z$
and so $\tilde{f}(x) = f(x)$ and $D\tilde{f}(x) = Df(x)$.

\begin{itemize}[itemsep=1ex]

\item
If $x \not\in \bigsqcup_{n=0}^{N-1} f^n(\bar Z)$
then  $\omega_j(x) = \omega_j(f(x)) = \gamma$, so \eqref{i.forward_invariance} follows from the invariance property \eqref{e.hor_cone_field_inv}.

\item 
If $x \in f^n(Z)$ for some $n \in \{1,\dots,N-2\}$ then 
\[
\omega_j(f(x)) \ge \omega_j(x) e^{-\chi_j(x)+\chi_{j+1}(x)+\epsilon} \, ,
\]
and so 
\eqref{i.forward_invariance} again follows from \eqref{e.hor_cone_field_inv}.

\item 
Finally, if $x \in f^{N-1}(Z)$ then, letting $y \coloneqq f^{-N+2}(x)$, we have:
\[
\omega_j(x) = 
\max \left\{ \gamma, \omega_j(y) \exp \sum_{i=0}^{N-2} \big[-\chi_j(f^i(y))+\chi_{j+1}(f^i(y))+\epsilon \big] \right\}
 \, .
\]
Using the fact that $\omega_j(y) \le \alpha$, condition \eqref{i.damping_gap} of the definition of damping perturbations, and the choice of $\epsilon$, we obtain that ${\omega_j(x) = \gamma}$.
Of course, $\omega_j(f(x))$ also equals $\gamma$, so 
\eqref{i.forward_invariance} again follows from \eqref{e.hor_cone_field_inv}.

\end{itemize}
This completes the proof of the \lcnamecref{l.new_cone}.
\end{proof}

Property \eqref{i.forward_invariance} from \cref{l.new_cone}
means that for each $j$, the continuous cone field $\cC_j(x) \coloneqq \cH_j(x,\omega_j(x))$ 
is strictly forward invariant.
This implies that the diffeomorphism $\tilde f$ has a simple dominated splitting $\tilde E_1 \oplus \cdots \oplus \tilde E_d$ such that for all $x$ and $j$, we have 
$\tilde{E}_1(x) \oplus \dots \oplus \tilde{E}_j(x) \subset \cC_j(x)$; see e.g.\ \cite[Proposition 2.2]{Samba}. 
In particular, conclusion \eqref{i.hor_perturb} of \cref{p.damping} is satisfied.
Since the definition of damping perturbations is symmetric under time-reversal, it follows that conclusion \eqref{i.ver_perturb} is also satisfied.

Next, consider the subbundles:
\begin{alignat*}{2}
E^\mathrm{u} &\coloneqq E_1 \oplus \dots \oplus E_u \, , &\quad 
E^\mathrm{s} &\coloneqq E_{u+1} \oplus \dots \oplus E_d \, , \\
\tilde E^\mathrm{u} &\coloneqq \tilde E_1 \oplus \dots \oplus \tilde E_u \, , &\quad 
\tilde E^\mathrm{s} &\coloneqq \tilde E_{u+1} \oplus \dots \oplus \tilde E_d \, .
\end{alignat*}
Let us check that the $D\tilde{f}$-invariant splitting $\tilde E^\mathrm{u} \oplus \tilde E^\mathrm{s}$ is uniformly hyperbolic.

Recall that $P^\mathrm{u} \colon TM \to TM$ denotes the projection onto $E^\mathrm{u}$ with kernel~$E^\mathrm{s}$.
Then, for every $x \in M$ and $v \in T_x M$, 
\begin{equation} \label{e.expansion_easy}
\tribar{P^\mathrm{u} Df(x) \, v} \ge e^{\chi_u(x)} \tribar{P^\mathrm{u}  v} \ge a_1 \tribar{P^\mathrm{u}  v} 
\, ,
\end{equation}
for some constant $a_1>1$.
Now consider $x \in Z$ and $v \in \cC_u(x)$; then: 
\begin{align*}
\tribar{P^\mathrm{u}  \, D\tilde{f}^N(x) v} 
&\ge \tribar{P^\mathrm{u} Df^{N-1}(\tilde{f}(x)) \, D\tilde{f}(x) \, v} \\
&\ge \exp\left[ \sum_{i=0}^{N-2} \chi_u(f^i(\tilde{f}(x))) \right] \tribar{P^\mathrm{u} D\tilde{f}(x) \, v} \\
&\ge e^{(N-1)\sigma/2} \tribar{P^\mathrm{u} D\tilde{f}(x) \, v} 
&\quad&\text{(by \eqref{i.damping_gap})} \\
&\ge \kappa e^{(N-1)\sigma/2}  \tribar{P^\mathrm{u} v} 
&\quad&\text{(by \eqref{i.damping_norm})} .
\end{align*}
Let $a_2 \eqcolon \kappa e^{(N-1)\sigma/2}$, which by inequality \eqref{e.def_N} is bigger than $1$.

Let $a_3 \coloneqq \min\{a_1,a_2^{1/N}\} > 1$.
We claim that 
there exists a constant $c > 0$ such that
for all $x\in M$, $v \in \cC_u(x)$, and $n \ge 0$,
\begin{equation}\label{e.projection_expansion}
\tribar{P^\mathrm{u}  \, D\tilde{f}^n(x) \, v} \ge c a_3^n \tribar{P^\mathrm{u}  \, v} \, .
\end{equation}
Indeed, the inequality holds with $c=1$ if the segment of orbit $\{x, \tilde{f}(x), \dots, \tilde{f}^{n-1}(x)\}$: 
\begin{itemize}
	\item either does not intersect $Z$;
	\item or, for each time it enters $Z$, it goes through the tower $Z \sqcup f(Z) \sqcup \cdots \sqcup f^{N-1}(Z)$ completely.
\end{itemize}
In general, the segment may end inside the tower; in this case its contribution is uniformly bounded, so we conclude that \eqref{e.projection_expansion} holds for an appropriate value of $c$.

Next, note that for all $x\in M$,
\begin{align*}
\cC_u(x) 
&= \big\{ v \in T_x M \st \tribar{P_\mathrm{s}(v)}^2 < \omega_u(x)^2  \tribar{P_\mathrm{u}(v)}^2  \big\} \cup \{0\} \\ 
&= \big\{ v \in T_x M \st \tribar{P_\mathrm{u}(v)}^2 > (1+\omega_u(x)^2)^{-1} \tribar{v}^2 \big\} \cup \{0\} \, .
\end{align*}
By \cref{l.new_cone}.\eqref{i.omega_bounds}, $\omega_u(x) \le \alpha$.
Using \eqref{e.projection_expansion} we conclude that for all $v \in \cC_u(x)$ and $n \ge 0$, 
\[
\tribar{D\tilde{f}^n(x) \, v} \ge 
\tribar{P^\mathrm{u}  \, D\tilde{f}^n(x) \, v} \ge 
c a_3^n \tribar{P^\mathrm{u}  \, v} \ge
c (1+\alpha^2)^{-1/2} a_3^n \tribar{v} \, .
\]
In particular, the bundle $\tilde E^\mathrm{u}$ is uniformly expanding.
By symmetry, the bundle $\tilde E^\mathrm{s}$ is uniformly contracting.
This shows that $\tilde f$ is an Anosov diffeomorphism of index $u$, completing the proof of \cref{p.damping}.
\end{proof}

\section{The model deformation}\label{s.model}

In this \lcnamecref{s.model} we define a special family of diffeomorphisms called ``model deformation'' that will be the basis for the construction in \cref{s.central_proof}. We also establish a few properties of those maps.

\medskip

Let $\| \mathord{\cdot} \|$ denote the Euclidean norm in $\R^d$,
and let $\B \coloneqq \{z \st \|z\| \le 1\}$ be the unit ball.
Denote by $m$ the Lebesgue measure on $\R^d$. 
All the constructions in this \lcnamecref{s.model} are in $\R^d$, so there is no risk of confusion with the volume measure on the manifold $M$.

Let $\Diff^\infty_m(\B,\partial\B)$ denote the set of all maps $h \colon \B \to \B$
that can be extended to  volume-preserving $C^\infty$-diffeomorphisms of $\R^d$
that coincide with the identity outside $\B$.

For each $j \in \{1,\dots,d-1\}$ and $\theta \in \R$, let $R^{(j)}_\theta$ be the orthogonal matrix that rotates the plane $\{0\}^{j-1} \times \R^2 \times \{0\}^{d-j-1}$ by angle $\theta$ and is the identity on its orthogonal complement, that is,
\begin{equation}\label{e.def_R_theta}
R^{(j)}_\theta \coloneqq
\begin{pmatrix}
\Id_{j-1}	&				&				&				\\
			& \cos \theta	& -\sin \theta	&				\\
			& \sin \theta	& \phantom{-}\cos \theta	&				\\
			&				&				& \Id_{d-j-1}
\end{pmatrix}
\, .
\end{equation}

Fix a rotationally symmetric (i.e.\ only depending on the norm) $C^\infty$-function $\rho \colon \R^d \to \R$ which is not identically zero and is supported on $\B$.
For each $j \in \{1,\dots,d-1\}$ and $s \in \R$, define $h^{(j)}_s \in \Diff^\infty_m(\B,\partial \B)$ by
\begin{equation}\label{e.def_h_elementary}
h^{(j)}_s(z) \coloneqq R^{(j)}_{s \rho(z)} (z) \, ;
\end{equation}
such a map is indeed conservative since it preserves the family of spheres with center at the origin and it acts on each of these as an orthonormal map.
The diffeomorphisms $h^{(j)}_s$ will be called the \emph{elementary model deformations}.
See the figure on \cite[p.~334]{ACW}.

For each $t = (t_1,\dots,t_{d-1}) \in \R^{d-1}$,
define $h_t \in \Diff^\infty_m(\B,\partial \B)$ by
\begin{equation}\label{e.def_ht}
h_t \coloneqq h^{(d-1)}_{t_{d-1}} \circ \cdots \circ h^{(1)}_{t_1} \, .
\end{equation}
This $d-1$-parameter family of diffeomorphisms will be called the \emph{model deformation}.

If $t \in \R^{d-1}$ is sufficiently close to $(0,\dots,0)$ then 
the diffeomorphism $h_t$ is sufficiently $C^1$-close to the identity so that 
the following \emph{transversality condition} holds:
for all $z \in \B$ and $j \in \{1,\dots,d-1\}$,
\begin{align}
\label{e.transv_hor}
Dh_t(z) (\R^j \times \{0\}^{d-j}) &\text{ is transverse to } \{0\}^j \times \R^{d-j} \, . 
\end{align}
Adjusting $\rho$ if necessary, we assume from now on that every $t$ in the unit cube $[0,1]^{d-1}$ satisfies the transversality condition.

\medskip

If $A$ is any $d \times d$ matrix and $I$, $J \subseteq \{1,\dots,d\}$ are nonempty subsets of the same cardinality, let us denote by $A[I,J]$ the submatrix formed by the entries with row in $I$ and column in $J$.
Recall that the determinants of those matrices are called the \emph{minors} of $A$;
the $k$-th \emph{principal minor} corresponds to $I = J = \{1,\dots,k\}$,
and is denoted by $\det\nolimits_k A$.

Note that the derivatives of the maps \eqref{e.def_h_elementary} have the following form (where asterisks are placeholders for unspecified entries):
\[
Dh_s^{(j)}(z) =
\begin{pmatrix}
1	&		&	&	&	&	&		&	\\
	&\ddots	&	&	&	&	&		& 	\\
	&		&1	&	&	&	&		&	\\
*	&\cdots	&*	&\circled{$*$}	&*	&*	&\cdots	&*	\\
*	&\cdots	&*	&*	&*	&*	&\cdots	&*	\\
	&		&	&	&	&1	&		&	\\
	&		&	&	&	&	&\ddots	&	\\
	&		&	&	&	&	&		&1 
\end{pmatrix}
\begin{matrix} 
    \coolrightbrace{1 \\ \ddots \\ 1 \\ \circled{$*$}}{j\phantom{d-j}} \\
    \coolrightbrace{* \\ 1 \\ \ddots \\ 1}{d-j\phantom{j}} 
\end{matrix}
\]
Moreover, since these matrices have unit determinant, we conclude that their principal minors are:
\begin{equation}\label{e.minor_factor}
\det\nolimits_k Dh_s^{(j)}(z) = 
\begin{cases}
	1 					&\text{if } k \neq j, \\
	\Delta^{(j)}_s(z)	&\text{if } k = j,
\end{cases}
\end{equation}
where $\Delta^{(j)}_s(z)$ is defined as the $(j,j)$-entry of the matrix $Dh^{(j)}_s(z)$, i.e.\ the circled entry above.

The matrices $Dh^{(j)}_s(z)$ have no common block triangular form, so the derivative of the composition \eqref{e.def_ht} is intricate. 
Nevertheless, there is a simple expression for its principal minors:

\begin{lemma}\label{l.minors}
For all $z \in \B$, $t = (t_1,\dots,t_{d-1}) \in \R^{d-1}$, and $j \in \{1,\dots,d-1\}$ we have:
\[
\det\nolimits_j Dh_t(z) = \Delta^{(j)}_{t_j} \circ h^{(j-1)}_{t_{j-1}} \circ \cdots \circ h^{(1)}_{t_1} (z) \, .
\]
\end{lemma}

\begin{proof}
Fix $z$ and $t$, and let $P \coloneqq Dh_t(z)$.
Let $z^{(0)} \coloneqq z$.
For each $k \in \{1,\dots, d-1\}$, let
$z^{(k)} \coloneqq h^{(k)}_{t_k} \circ \cdots \circ h^{(1)}_{t_1} (z)$
and $A_k \coloneqq Dh^{(k)}_{t_k} (z^{(k-1)})$.
By the chain rule, $P = A_{d-1} \cdots A_2 A_1$.

Fix $j \in \{1,\dots,d-1\}$, and let $J \coloneqq \{1,\dots,j\}$.
We claim that if $I \subset \{1,\dots,d\}$ has cardinality $j$ and $I \neq J$ then:
\begin{align}
\label{e.IJ}
\det A_k [I,J] = 0 \text{ for every } k &< j, \text{ and }\\
\label{e.JI}
\det A_k [J,I] = 0 \text{ for every } k &> j. 
\end{align}
Indeed, for all $i>j$, 
the row matrix $A_k[\{i\},J]$ vanishes if $k<j$, 
and the column matrix $A_k[J,\{i\}]$ vanishes if $k>j$. 

By the Cauchy--Binet formula for the minors of a product, 
\[
\det P[J,J] = 
\sum \det A_{d-1}[J,I_{d-2}] \cdots \det A_j [I_j,I_{j-1}] \cdots \det A_1 [I_1,J] \, , 
\]
where the sum is over all $d-2$-tuples $(I_1,\dots,I_{d-2})$ of subsets of $\{1,\dots,d\}$ of cardinality $j$.
Consider a nonzero term of this sum.
Using \eqref{e.IJ} recursively we obtain $J = I_1 = I_2 = \cdots = I_{j-1}$.
On the other hand, using \eqref{e.JI} recursively we obtain $J = I_{d-2} = I_{d-3} = \cdots = I_j$.
Therefore the sum contains a single nonzero term, which by \eqref{e.minor_factor} equals $\Delta^{(j)}_{t_j}(z^{(j-1)})$.
\end{proof}

The transversality condition \eqref{e.transv_hor} is equivalent to the non-vanishing of the minors $\det\nolimits_j Dh_t(z)$. 
When $t=0$ these minors equal $1$.
So the minors are always strictly positive.
In particular, by \cref{l.minors}, the functions $\Delta^{(j)}_s$ are strictly positive for $s \in [0,1]$.
We will need information about the logarithms of these functions.
By definition,
\[
\Delta^{(j)}_s(z_1,\dots,z_d) 
=
\frac{\partial}{\partial z_j} \Big( z_j \cos (s\rho(z)) - z_{j+1} \sin (s\rho(z)) \Big) \, .
\]
Since $\rho$ is rotationally symmetric, 
\[
\Delta^{(j)}_s (z_1,\dots,z_d) = \Delta^{(1)}_s (z_j, z_{j+1}, \dots, z_d, z_1, \dots, z_{j-1}).
\]
In particular, for each $s \in [0,1]$, the integral
\begin{equation}\label{e.def_Q}
{\mathcal Q}(s) \coloneqq - \frac{1}{m(\B)} \int_{\B} \log \Delta^{(j)}_s  \dd m \, ,
\end{equation}
is independent of $j$.

\begin{lemma}\label{l.Jensen}
For every $s \in [0,1]$ we have ${\mathcal Q}(s) \ge 0$, with equality if and only if $s=0$.
\end{lemma}

\begin{proof}
Consider $j=1$ in \eqref{e.def_Q}.
Then the lemma follows from Jensen inequality and a geometric argument: see \cite[Lemma~1.2]{BarBon}.
\end{proof}

Since each $h^{(j)}_{t_j}$ is volume preserving, we obtain from \eqref{e.def_Q} and \cref{l.minors} that
\begin{equation}\label{e.int_log_minor}
\frac{1}{m(\B)} \int_{\B} \log \det\nolimits_j Dh_t \dd m = -{\mathcal Q}(t_j) \, .
\end{equation}
This formula will be essential for our deformations of Lyapunov exponents: roughly speaking, it will allow us to move the summed exponents $\hat\lambda_j$ simultaneously and independently, as we will see in the next section.  

\medskip

Let us discuss a few other features of the model deformation. 

\medskip

Recall the definitions \eqref{e.hor_cone}, \eqref{e.ver_cone} of standard cones in $\R^d$.
By the transversality assumption \eqref{e.transv_hor}, 
if $\alpha > 1$ is large enough and $\beta \coloneq \alpha^{-1}$ then 
for all $z \in \B$, $t \in [0,1]^{d-1}$, and $j \in \{1,\dots,d-1\}$, we have:
\begin{equation}\label{e.nested_cones}
Dh_t(z) \cH_j(\beta) \subseteq \cH_j(\alpha) \, \quad \text{and} \quad
Dh_t^{-1}(z) \cV_j(\beta) \subseteq \cV_j(\alpha) \, .
\end{equation}
In fact the inclusions are equivalent to one another since the complement of a horizontal cone $\cH_j(\tau)$ is (modulo sets of non-empty interior) a vertical cone $\cV_j(\tau)$.
Let us fix these numbers $\alpha > \beta > 0$ from now on.

Let $P_j$ and $Q_j \colon \R^d \to \R^d$ be the projections on the spaces $\R^j \times \{0\}^{d-j}$ and $\{0\}^j \times \R^{d-j}$, respectively.
We fix $\kappa > 0$ such that
\begin{equation}\label{e.kappa}
\left\{
\begin{array}{ccc}
v \in \cH_j(\beta) & \Rightarrow & \; \; \| P_j \, Dh_t(z) \, v \| \ge \kappa \, \|P_j \, v \| \, , \\
v \in \cV_j(\beta) & \Rightarrow &  \| Q_j \, Dh_t^{-1}(z) \, v \|  \ge \kappa \, \|Q_j \, v\|  \, .
\end{array}
\right.
\end{equation}

\medskip

Let us recall a few facts about exterior powers. 
Let $\{e_1,\dots,e_d\}$ denote the canonical basis in $\R^d$.
Then, for each $k \in \{1,\dots,d\}$, $\wed^k \R^d$ is a vector space with the \emph{canonical} basis $\{e_{i_1} \wedge \dots \wedge e_{i_k}\}_{i_1 < \dots < i_k}$.
We endow $\wed^k \R^d$ with the inner product that makes the canonical basis orthonormal. 
There is a \emph{wedge} operation $(v,w) \in (\wed^k \R^d) \times (\wed^\ell \R^d) \mapsto v \wedge w \in \wed^{k+\ell} \R^d$ which is associative, mutilinear, and skew-symmetric (i.e., $w \wedge v = (-1)^{k\ell} v \wedge w$).
Given a linear map $A \colon \R^d \to \R^d$, there are induced linear maps $\wed^k A \colon \wed^k \R^d \to \wed^k \R^d$ such that 
$(\wed^k A) (e_{i_1} \wedge \dots \wedge e_{i_k}) = Ae_{i_1} \wedge \dots \wedge Ae_{i_k}$.
The entries of the matrix of the exterior power $\wed^k A$ with respect to the canonical basis are exactly the $k \times k$ minors of $A$; this is called the \emph{$k$-th compound matrix} of $A$. More precisely, if $I=\{i_1<\dots<i_k\}$ and $J=\{j_1<\dots<j_k\}$ then
\[
\left\langle e_{i_1} \wedge \dots \wedge e_{i_k} , \, (\wed^k A)(e_{j_1} \wedge \dots \wedge e_{j_k}) \right\rangle = \det A[I,J] \, .
\]
So the Cauchy--Binet formula used before is nothing but the functoriality of the exterior powers, i.e., $\wed^k(AB) = (\wed^k A)(\wed^k B)$.
If $A$ is an orthogonal linear map, then so is $\wed^k A$.
Therefore, the inner product on $\wed^k \R^d$ depends only on the inner product on $\R^d$, and not on the choice of orthonormal basis; equivalently, the $k$-fold exterior power of an inner product space is an inner product space.
From a more geometrical viewpoint, if $v_1$, \dots, $v_k$ are vectors in $\R^d$ then the norm of $v_1 \wedge \dots \wedge v_k$ is the $k$-dimensional volume of the parallelepiped spanned by these vectors.

\medskip 

Coming back to the model deformation, let us note for later use that for any $\nu>0$ there exists $\gamma \in (0,\beta)$ such that for all $z \in \B$, $t \in [0,1]^{d-1}$, $j \in \{1,\dots,d-1\}$ and all linearly independent vectors $w_1$, \dots, $w_j \in \cH_j(\gamma)$, 
we have
\begin{equation}\label{e.minor_cone}
\left| \log \frac{\left\langle (\wed^j Dh_t(z))(w_1 \wedge \dots \wedge w_j) , e_1 \wedge \dots \wedge e_j \right\rangle}{\left\langle w_1 \wedge \dots \wedge w_j , e_1 \wedge \dots \wedge e_j \right\rangle} - \log \det\nolimits_j Dh_t(z) \right| < \nu \, ;
\end{equation}
this statement includes the fact that the numerator and the denominator in the expression above are both  nonzero and have the same sign.

\section{Proof of the central proposition}\label{s.central_proof}

We now assemble the material of the previous three sections in order to prove \cref{p.central} (which, as we have seen, implies \cref{t.majorize}). Recall that we need to define a multiparametric perturbation $(f_t)$ of $f$ whose Lyapunov exponents ``spread out'' in a controlled way (see \cref{f.shaky_cube}). Our diffeomorphisms $f_t$ will be damping perturbations of the given $f$. Recall from \cref{ss.damping_properties} that the damping time $N$ (that is, the minimal return time to the support of the perturbation) must be roughly proportional to $1/\sigma$, where $\sigma$ is the bound on the gap. By preservation of the measure, the volume of the support of the perturbations cannot be larger than $1/N \asymp \sigma$. This upper bound can be essentially attained using Rokhlin tower lemma to select convenient places to perturb. 
The actual maps $f_t$ will be defined by inserting copies of our model deformations on many ``Lyapunov balls'' (i.e., balls with respect to the Lyapunov charts constructed in \cref{ss.adapted_charts}). Our preparations will ensure that such maps $f_t$ are Anosov with simple dominated splittings $E_{1,t} \oplus \cdots \oplus E_{d,t}$, and what we are left to do is to estimate their Lyapunov exponents. It is convenient to work with the summed exponents $\hat{\lambda}_j(f_t)$, which equal the average log of the expansion of $j$-dimensional volume along the invariant subbundle $E_{1,t} \oplus \cdots \oplus E_{j,t}$. It boils down that, modulo a small error\footnote{This error is due to the fact that the invariant subbundles are not in general not smoothly integrable, and so the Lyapunov charts ``straighten'' these bundles only approximately. But this ``noise'' can be made arbitrarily small by reducing the radii of the Lyapunov balls.}, $\hat{\lambda}_j(f_t)$ drops as compared to the unperturbed $\hat{\lambda}_j(f)$, and the drop is proportional to two things (see \eqref{e.upshot}): the integrals \eqref{e.int_log_minor} that control the basic effect of each model perturbation, and the total measure of the supports of the perturbations. The integrals are bounded by constants, and the measure of the  support of the whole damping perturbation is, as explained above, of the order of $\sigma$. So ultimately we are able to move Lyapunov exponents by an amount that depends on the gap $\sigma$ but not on $f$ itself, which is an essential feature of \cref{p.central}.

\begin{proof}[Proof of \cref{p.central}]
Let $u \in \{1,\dots,d-1\}$ and $a_1$, \dots, $a_{d-1}$, $\sigma$, $\delta_0 > 0$ be given.
Recall that ${\mathcal Q}(s)$ denotes the quantity defined in \eqref{e.def_Q}, which is a continuous function of $s \in [0,1]$, vanishes at $s=0$, and is positive for $s>0$ (by \cref{l.Jensen}).
Reducing $\delta_0$ if necessary, we may assume that
\[
3 \delta_0 \max\{a_1,\dots,a_{d-1}\} \le {\mathcal Q}(1) \, .
\]
For each $j \in \{1,\dots,d-1\}$, 
define:
\begin{equation}\label{e.def_bj}
b_j \coloneqq \text{ the least value in $[0,1]$ such that } {\mathcal Q}(b_j) = 3 \delta_0 a_j \, .
\end{equation}
We fix several other constants. 
Let $\alpha > \beta > 0$ be openings with property \eqref{e.nested_cones}.
Let $\kappa > 0$ be a constant with property \eqref{e.kappa}.
Let $\nu \coloneqq (\delta_0/2) \min\{a_1,\dots,a_{d-1}\}$
and take $\gamma \in (0,\beta)$ with property \eqref{e.minor_cone}.
Let $N \coloneqq N_0(\alpha, \beta, \gamma, \kappa, \sigma)$ be given by \eqref{e.N0}.
Let $\delta \coloneqq \delta_0 / N$; this will be the scaling factor as in the statement of \cref{p.central}.
Choose a very small $\epsilon>0$; this choice will be apparent at the end of the proof.

Now let us pick $f \in \Diff_m^\reg(M)$ as in the statement of \cref{p.central}, that is, a conservative Anosov $C^\reg$-diffeomorphism of unstable index $u$, admitting a simple dominated splitting, and such that $\mathsf{g}_u(\boldsymbol{\lambda}(f)) \ge \sigma$.

We apply \cref{p.adapted_metric} and find a Lyapunov metric so that the associated expansion functions $\chi_1$, \dots, $\chi_d$ obey the $L^1$-estimate \eqref{e.adapted_L1} with the chosen value of $\epsilon$.
Consider the sets 
\[
R_j \coloneqq \big\{x \in M \st | \chi_j(x) - \lambda_j(f)| \ge \sigma/2 \big\} \, .
\]
By a trivial estimate (Markov's inequality),
$m(R_j) \le 2 \sigma^{-1} \epsilon$.
Thus, the open set
\[
U \coloneqq M \setminus  \bigcup_{j=1}^{d-1}  \bigcup_{n=-N+1}^{N-1}  f^n(R_j) 
\]
has measure $m(U) > 1 - 4 d N \sigma^{-1}\epsilon$.

By the Rokhlin Lemma, there exists a measurable set $Z_1$ disjoint from $f(Z_1)$, $f^2(Z_1)$, \dots, $f^{N-1}(Z_1)$ with measure 
\[
m(Z_1) > \frac{1}{N} - \epsilon \, .
\]
Then 
\[
m (Z_1 \cap U) > \frac{1}{N} - (1 + 4 d N \sigma^{-1}) \epsilon \, .
\]
Let $Z_2$ be a compact subset of $Z_1 \cap U$ satisfying the same bound.
Then take an open neighborhood $Z_3 \subseteq U$ of $Z_2$ that satisfies the same bound, and, moreover, is disjoint from $f(Z_3)$, $f^2(Z_3)$, \dots, $f^{N-1}(Z_3)$.

Let $\Phi_x \colon B_0 \to M$ be the Lyapunov charts coming from \cref{p.adapted_charts};
here $B_0$ is a rescaling of the unit closed ball $\B$, that is $B_0 = s_0 \B$ for some $s_0>0$.
For every $x \in M$ and every $s \in (0,s_0]$, let us define a \emph{Lyapunov ball} as
\[
\cB(x,s) \coloneqq \Phi_x ( s \B) \, .
\]
Reducing $s_0$ if necessary, we assume that the following properties hold:
\begin{align}
\label{e.small_ball_chi}
y, y' \in \cB(x,s) \ &\Rightarrow \  |\chi_j(y) - \chi_j(y')| < \epsilon \, , \\
\label{e.small_ball_Phi}
y     \in \cB(x,s) \ &\Rightarrow \  \big\| D\Phi_x^{-1}(y) \cL_y  - \Id \big\| < \epsilon \, , 
\end{align}
where $\cL_y \coloneqq D\Phi_y(0)$.

Since the Lyapunov charts $\Phi_x$ form a relatively compact subset of $C^1(B_0,M)$, 
the family of Lyapunov balls contained in the open set $Z_3$ forms a Vitali covering of $Z_3$.
Therefore, there exists a finite collection of disjoint Lyapunov balls $\cB(x_1,s_1)$, \dots, $\cB(x_p, s_p) \subseteq Z_3$
whose union $Z_4$ has measure $m(Z_4) > m(Z_3) - \epsilon$. 
To simplify notation, 
let $\cB_i \coloneqq \cB(x_i,s_i)$ and define
\begin{equation}\label{e.def_Psi}
\Psi_i \colon \B \to M \quad \text{by} \quad
\Psi_i(z) \coloneqq \Phi_{x_i}(s_i z) \, ;
\end{equation}
so $\Psi_i$ has constant Jacobian and image $\cB_i$.

For each $t \in [0,1]^{d-1}$, let us define $g_t \in \Diff_m^\infty(M)$ as follows:
$g_t$ equals the identity outside of $Z_4$, and on each $\cB_i$ it is defined as:
\begin{align}\label{e.def_g}
g_t \coloneqq \Psi_i \circ h_{Bt} \circ \Psi_i^{-1} \, , 
\end{align}
where $B(t_1,\dots,t_{d-1}) \coloneqq (b_1 t_1, \dots, b_{d-1} t_{d-1})$, and $b_j$ comes from \eqref{e.def_bj}.
The sought-after deformation 
of $f$ is defined as:
\[
f_t \coloneqq f \circ g_t \, , \quad \text{for } t \in [0,1]^{d-1} \, .
\] 

We now need to check that the maps $f_t$ have the properties asserted in \cref{p.central}.
As a first step, we show:

\begin{lemma}\label{l.checker}
Each $f_t$ is a $(\alpha,\beta,\kappa,\sigma,N)$-damping perturbation of $f$.
\end{lemma}

\begin{proof}
Condition~\eqref{i.damping_cones}  in the definition of damping perturbations follows from the inclusions~\eqref{e.nested_cones} and the fact that $Df$ (resp.\ $Df^{-1}$) decreases the opening of horizontal (resp.\ vertical) cones: see \eqref{e.hor_cone_field_inv}, \eqref{e.ver_cone_field_inv}.
Let us check condition~\eqref{i.damping_norm}.
If $v \in \cH_u(x,\beta)$ then, using property~\eqref{e.kappa},
\begin{align*}
\tribar{P^\mathrm{u} D\tilde{f}(x) \, v}
= \tribar{P^\mathrm{u} Df(g_t x) \, Dg_t(x) \, v} 
\ge \tribar{P^\mathrm{u} Dg_t(x) \, v} 
\ge \kappa \, \tribar{P^\mathrm{u} v}  \, , 
\end{align*}
while if $v \in \cV_u(x,\beta)$ then $D\tilde{f}^{-1} (x) \, v \in \cV_u(\tilde{f}^{-1} (x) ,\beta)$ and so:
\begin{align*}
\tribar{P^\mathrm{s} D\tilde{f}^{-1} (x) \, v}
= \tribar{P^\mathrm{s} Dg_t^{-1}(\tilde{f}^{-1} (x)) D\tilde{f}^{-1} (x) \, v} 
\ge \kappa \, \tribar{P^\mathrm{s} D\tilde{f}^{-1} (x) \, v}
\ge \kappa \, \tribar{P^\mathrm{s} v}  \, .
\end{align*}

Condition~\eqref{i.damping_tower} follows from the fact that the set $Z_4$ is disjoint from its $N-1$ first iterates, while condition~\eqref{i.damping_gap} follows from the fact that $Z_4$ is contained in $U$.
\end{proof}

Now \cref{p.damping} ensures that each $f_t$ is an Anosov diffeomorphism of unstable index $u$ and admits a simple dominated splitting $TM = E_{1,t} \oplus \dots \oplus E_{d,t}$.
Moreover, by part \eqref{i.hor_perturb} of the \lcnamecref{p.damping},
\begin{align}
\label{e.robust_transv}
&E_{1,t}(x) \oplus \dots \oplus E_{j,t}(x)
\text{ is transverse to }
E_{j+1}(x) \oplus \dots \oplus E_d(x) \, , \\
\label{e.recover}
&E_{1,t}(x) \oplus \dots \oplus E_{j,t}(x) \subset \cH_j( x, \gamma) \quad \text{if $x \in Z_4$.}
\end{align}

The rest of the proof consists in estimating the summed Lyapunov exponents $\hat{\lambda}_j(f_t)$.

For each $x \in M$ and $j \in \{1,\dots,d\}$, let $v_j(x) \coloneqq \cL_x e_j$, which is a vector that spans $E_j(x)$.
Since these vectors form an orthonormal basis of $T_x M$, using the definition \eqref{e.def_chi} of the expansion function $\chi_j$, we obtain that for every vector $w \in T_x M$,
\begin{equation}\label{e.train1}
\biangle{Df(x) w, v_j(f(x))} = e^{\chi_j(x)} \biangle{w,v_j(x)} \, .
\end{equation}

In each exterior power of the tangent bundle we consider the corresponding exterior power of the Lyapunov metric, denoting it by the same symbol.
Let
\[
\hat{v}_{j}(x)   \coloneqq v_{1}(x) \wedge \dots \wedge v_{j}(x) \in \wed^j T_x M \, .
\]
Note that for every $\hat{w} \in \wed^j T_x M$,
\begin{equation}\label{e.train2}
\biangle{\wed^j Df(x) \hat w, \hat{v}_j(f(x))} = e^{\chi_1(x) + \cdots + \chi_j(x)} \biangle{\hat w,\hat{v}_j(x)} \, .
\end{equation}

For each $x\in M$, $j \in \{1,\dots,d\}$, and $t \in [0,1]^{d-1}$, let us choose a vector $v_{j,t}(x)$ that spans $E_{j,t}(x)$ in such a way that it depends continuously on $t$ and equals $v_j(x)$ when $t=0$.
Let
\[
\hat{v}_{j,t}(x) \coloneqq v_{1,t}(x) \wedge \dots \wedge v_{j,t}(x) \in \wed^j T_x M \, .
\]
Define the functions
\begin{equation}\label{e.def_psi}
\psi_{j,t}(x) \coloneqq 
\frac{\left| \biangle{ \wed^j Df_t(x) \, \hat v_{j,t}(x) ,  \hat v_j(f_t(x))} \right|}
{\biangle{\hat v_{j,t}(x) , \hat v_j(x)}} 
 \, .
\end{equation}
Note that the numerator and the denominator in this formula are both positive, thanks to \eqref{e.robust_transv}.

\begin{lemma}\label{l.formula_exponent}
$\displaystyle \int_M \log \psi_{j,t} \dd m = \hat{\lambda}_j(f_t)$.
\end{lemma}

\begin{proof}
The top Lyapunov exponent of the linear cocycle $\wed^j Df_t$ equals $\hat{\lambda}_j(f_t)$, has multiplicity $1$, and the corresponding Oseledets space at a regular point $x$ is spanned by $\hat v_{j,t}(x)$ (see \cite[Theorem~5.3.1]{LArnold}).
In particular, the corresponding expansion function
\[
\log \frac
{\tribar{ \wed^j Df_t(x) \, \hat v_{j,t}(x)}} 
{\tribar{\hat v_{j,t}(x)}} \eqcolon \log \tilde \psi_{j,t}(x) 
\]
has integral $\hat{\lambda}_j(f_t)$.
Furthermore, we have: 
\[
\wed^j Df_t(z) \, \hat v_{j,t}(x) = 
\pm \tilde \psi_{j,t}(x) \, \hat v_{j,t}(f_t(x)) \, .
\]
Substituting into \eqref{e.def_psi}, we see that the functions $\log \psi_{j,t}$ and $\log \tilde \psi_{j,t}$ are cohomologous.
In particular they have the same integral, proving the \lcnamecref{l.formula_exponent}.
\end{proof}

Let us analyze the function $\log \psi_{j,t}$ in order to estimate its integral.

\begin{lemma}\label{l.no_change}
If $x \not\in Z_4$ then $\log \psi_{j,t}(x) = \chi_1(x) + \dots + \chi_j(x)$.
\end{lemma}

\begin{proof}
If $x\notin Z_4$ then $f_t=f$ on a neighborhood of $x$ and so the \lcnamecref{l.no_change} follows from \eqref{e.train2}.
\end{proof}

On the other hand, on $Z_4 = \bigsqcup_i \cB_i$ we have the following estimate:

\begin{lemma}\label{l.key_estimate}
If $x \in \cB_i$ and $z \coloneqq \Psi_i^{-1}(x)$ then:
\begin{equation} \label{e.key_estimate}
\Big| \log \psi_{j,t}(x) - \big[ \chi_1(x) + \dots + \chi_j(x) + \log \det\nolimits_j Dh_{Bt}(z) \big] \Big| \le \nu + O(\epsilon) \, .
\end{equation}
\end{lemma}

Here and in what follows, $O(\epsilon)$ stands for anything whose absolute value is bounded by $\epsilon$ times something depending on the numbers fixed at the beginning of the proof and on the model deformation. 

\begin{proof}
Consider the case $j=1$.
Since $f_t = f \circ g_t$, we have:
\begin{align*}
\log \psi_{1,t}(x) &= 
\log \left|
\frac{\biangle{ Df(g_t(x)) Dg_t(x) v_{1,t}(x)  \, ,  v_1(f(g_t(x)))}}
{\biangle{ Dg_t(x) v_{1,t}(x) \, ,  v_1(g_t(x))}}
\right|  \\
& \hspace{.3\textwidth}
+ \log \left|
\frac{\biangle{ Dg_t(x) v_{1,t}(x) ,  v_1(g_t(x))}}
{\biangle{v_{1,t}(x) , v_1(x)}} 
\right| \\
&\eqcolon \circled{1} + \circled{2}.
\end{align*}
The first term is easy to estimate, using \eqref{e.train1} and \eqref{e.small_ball_chi}:
\begin{equation} \label{e.circled1}
\circled{1} = \chi_1(g_t(x)) = \chi_1(x) + O(\epsilon)
\end{equation}

Consider $x$ as fixed and denote $e_{1,t} \coloneqq \cL_x^{-1} v_{1,t}(x)$.
Then:
\[
\circled{3} \coloneqq
\biangle{v_{1,t}(x) , v_1(x)} = \biangle{\cL_x e_{1,t} , \cL_x e_1} = \langle e_{1,t}, e_1 \rangle,
\]
since $\cL_x$ sends the Euclidean metric on $\R^d$ to the Lyapunov metric at $T_x M$.
Recall from \eqref{e.recover} that $v_{1,t}(x) \in \cH_1(x,\gamma)$, that is, $e_{1,t} \in \cH_1(\gamma)$.
So the inner product $\circled{3}$ is always positive (by continuity) and actually satisfies the bound: 
\begin{equation}\label{e.3_lower_bound}
\circled{3} 
\ge (1+\gamma^2)^{-1/2}\|e_{1,t}\| \, .
\end{equation}
Recall that $z \coloneqq \Phi_i^{-1}(x)$.
By \eqref{e.minor_cone}, $\langle  Dh_{Bt}(z) e_{1,t} , e_1 \rangle$ is positive and 
\begin{equation}\label{e.train3}
\left| \log \frac{\langle  Dh_{Bt}(z)  e_{1,t} , e_1 \rangle}{\langle e_{1,t} , e_1 \rangle}  - \log \det\nolimits_1 Dh_{Bt}(z) \right| < \nu.
\end{equation}
Next, consider the quantity
\begin{multline*}
\circled{4} \coloneqq
\biangle{ Dg_t(x) v_{1,t}(x) ,  v_1(g_t(x))} \\ =  \biangle{ Dg_t(x) \cL_x e_{1,t} ,  \cL_{g_t(x)} e_1 } 
 = \langle \cL_{g_t(x)}^{-1} Dg_t(x) \cL_x e_{1,t} , e_1\rangle \, .
\end{multline*}
Applying the chain rule,
\begin{alignat*}{2}
Dg_t(x)
&= D\Psi_i(h_{Bt}(z)) \, Dh_{Bt}(z) \, D\Psi_i^{-1}(x) &\qquad&\text{(by \eqref{e.def_g})} \\
&= D\Phi_{x_i}(\Phi_{x_i}^{-1}(g_t(x))) \, Dh_{Bt}(z) \, D\Phi_{x_i}^{-1}(x)  &\qquad&\text{(by \eqref{e.def_Psi}).} 
\end{alignat*}
Therefore,
\begin{align*}
\circled{4} 
&= 
\Big\langle \big[ D\Phi_{x_i}^{-1}(g_t(x)) \cL_{g_t(x)}\big]^{-1} \, Dh_{Bt}(z) \, \big[D\Phi_{x_i}^{-1}(x) \cL_x\big]  e_{1,t}  \, , e_1 \Big\rangle  \\
&\eqcolon 
\big\langle  (\Id+P)^{-1} \, Dh_{Bt}(z) \, (\Id+Q)  e_{1,t}  \, , e_1 \big\rangle  \, ,
\end{align*}
where, by \eqref{e.small_ball_Phi}, the linear maps $P, Q \colon \R^d \to \R^d$ satisfy $\|P\|$, $\|Q\| < \epsilon$.
So
\begin{equation}\label{e.circled4}
\circled{4}  = \langle  Dh_{Bt}(z)  e_{1,t} , e_1 \rangle + O(\epsilon\|e_{1,t}\|) \, ,
\end{equation}
and in particular $\circled{4}$ is positive.
Now we can estimate:
\begin{alignat*}{2}
\circled{2} 
&= \log \left( \circled{4} / \circled{3} \right)
&\quad&\text{(by definition)} 
\\
&= \log \left( \frac{\langle  Dh_{Bt}(z)  e_{1,t} , e_1 \rangle}{\langle e_{1,t} , e_1 \rangle} + O(\epsilon) \right)
&\quad&\text{(using \eqref{e.3_lower_bound} and \eqref{e.circled4})} \\ 
&= \log \left( \frac{\langle  Dh_{Bt}(z)  e_{1,t} , e_1 \rangle}{\langle e_{1,t} , e_1 \rangle}  \right) + O(\epsilon) 
&\quad&\text{(using \eqref{e.train3}).} 
\end{alignat*}
Combining this with \eqref{e.train3} and \eqref{e.circled1}, we obtain the desired estimate \eqref{e.key_estimate} for the case $j=1$.

The proof for $j\ge 2$ follows exactly the same pattern, taking exterior powers, of course.
The only point that deserves notice is that estimate \eqref{e.3_lower_bound} should be replaced by the following:
if 
$e_{j,t} \coloneqq \cL_x^{-1} v_{j,t}(x)$ then
\[
|\langle e_{1,t} \wedge \dots \wedge e_{j,t}, e_1 \wedge \dots \wedge e_j \rangle| \ge (1+\gamma^2)^{-j/2}\| e_{1,t} \wedge \dots \wedge e_{j,t} \| \, .
\]
Indeed, the orthogonal projection onto the space spanned by $e_1$, \dots, $e_j$ cannot contract a vector in the cone $\cH_j(\gamma)$ by a factor smaller than $(1+\gamma^2)^{-1/2}$, and hence it cannot contract the volume of a $j$-dimensional parallelepipid in the cone by a factor smaller than $(1+\gamma^2)^{-j/2}$.
\end{proof}

We obtain from \eqref{e.key_estimate} and \eqref{e.int_log_minor} that for each Lyapunov ball $\cB_i$,
\[
\left|
\int_{\cB_i} \log \psi_{j,t} \dd m 
- \int_{\cB_i} (\chi_1 + \dots + \chi_j) \dd m 
+ {m(\cB_i)} {\mathcal Q}(b_j t_j)
\right|
\le {m(\cB_i)} \nu + O(\epsilon) \, .
\]
This together with \cref{l.no_change,l.formula_exponent} yields:
\begin{equation}\label{e.upshot}
\big| \hat\lambda_j(f_t) - \hat\lambda_j(f) +  m(Z_4) {\mathcal Q}(b_j t_j) \big| \le m(Z_4) \nu + O(\epsilon)  \, .
\end{equation}
Since $\nu \le \delta_0 a_j / 2$ 
and $1/N - m(Z_4) = O(\epsilon)$, we have:
\[
\left| \hat\lambda_j(f_t) - \hat\lambda_j(f) + \frac{{\mathcal Q}(b_j t_j)}{N} \right| \le \frac{\delta_0 a_j}{2N} + O(\epsilon)  \, .
\]
By the definition \eqref{e.def_bj} of $b_j$ we have:
\[
{\mathcal Q}(0) = 0 \le {\mathcal Q}(b_j t_j) \le 3 \delta_0 a_j = {\mathcal Q}(b_j) \, .
\]
Combining these two pieces of information, and recalling that $\delta = \delta_0/N$, we obtain:
\begin{align*}
\hat\lambda_j(f_t) - \hat\lambda_j(f) 
&\ge 
\begin{cases} 
-3.5 \, \delta a_j - O(\epsilon) &\text{for all $t$;} \\
-0.5 \, \delta a_j - O(\epsilon) &\text{if $t_j=0$;}
\end{cases}
\\
\hat\lambda_j(f_t) - \hat\lambda_j(f) 
&\le 
\begin{cases} 
\phantom{+}0.5 \, \delta a_j + O(\epsilon) &\text{for all $t$;} \\
         - 2.5 \, \delta a_j + O(\epsilon) &\text{if $t_j=1$.}
\end{cases}
\end{align*}
So a sufficiently small choice of $\epsilon$ ensures that inequalities \eqref{e.box_bounds}--\eqref{e.box_bottom} are satisfied.
\Cref{p.central} and, therefore, \cref{t.majorize} are proved.
\end{proof}

\section{Proof of additional results for tori}\label{s.torus}

In this \lcnamecref{s.torus} we prove two supplements to \cref{t.majorize}, namely \cref{c.easy,t.T3}.

\subsection{Anosov diffeomorphisms display all hyperbolic simple Lyapunov spectra}\label{ss.easy_proof}

In order to deduce \cref{c.easy} from \cref{t.majorize}, we need the following fact:

\begin{lemma}\label{l.number_theory}
Given integers $d>u>0$, there exists an Anosov linear automorphism of $\T^d$ with unstable index $u$ and simple Lyapunov spectrum.
\end{lemma}

\begin{proof}
We need to show the existence of a polynomial $P(x)$ with integer coefficients, leading term $x^d$, constant term $\pm 1$, whose roots are all real and simple, being $u$ of them with modulus bigger than~$1$ and $d-u$ of them of modulus smaller than~$1$. Then the companion matrix of $P(x)$ will be an element $L \in \GL(d,\Z)$ that induces an Anosov linear automorphism $F_L \colon \T^d \to \T^d$ with unstable index~$u$ and simple Lyapunov spectrum.
Though the existence of such polynomials can be quickly deduced from Dirichlet's unit theorem, we will provide a completely elementary proof. The idea comes from a proof of existence of Pisot--Vijayaraghavan numbers of arbitrary degree \cite[Theorem~1]{Vija}.

Fix integers 
\[
a_1>\cdots>a_u>0>a_{u+1}>\cdots>a_d
\]
whose sum is $0$ and such that $a_i-a_{i+1}\ge 2$ for each $i\in\{1,\dots,d-1\}$. 
Let $b \ge 3$ be another integer, and consider the polynomial
\[
P(x) \coloneq \sum_{i=0}^d (-1)^i b^{\hat{a}_i} x^{d-i} 
\quad\text{where}\quad 
\hat{a}_i \coloneq a_1+\cdots+a_i, \quad \hat{a}_0 \coloneqq 0.
\]
Note that $\hat{a}_i \ge 0$ for each $i$ and so $P$ has integer coefficients. 
Furthermore, $P$ is monic and has constant term $(-1)^d$. 

Let us locate the roots of $P$.
We claim that, for all $n \in \Z \setminus \{a_1,\dots,a_d\}$,
\begin{equation}\label{e.sign_claim}
P(b^n) \text{ is non-zero and has the same sign as } \prod_{j=1}^d (n-a_j) \, ,
\end{equation}
and so, by the intermediate value theorem, $P$ has $d-u$ simple roots on the interval $(0,1)$ and $u$ simple roots on the interval $(1,+\infty)$.
In order to prove the claim, fix $n$ and consider the expression:
\begin{equation}\label{e.tricky_sum}
P(b^n) = \sum_{i=0}^d (-1)^i b^{\hat{a}_i + n(d-i)} \, .
\end{equation}
The function $i \in \{0,\dots,d\} \mapsto \hat{a}_i + n(d-i)$ is integer-valued, strictly concave, and attains a maximum at $i=k$, where $k$ is the number of negative factors in the product $\prod_{j=1}^d (n-a_j)$. Using that $2\sum_{j=1}^\infty b^{-j} \le 1$, we see that the term corresponding to $i=k$ in the right-hand side of \eqref{e.tricky_sum} is bigger in absolute value than the sum of all other terms. So the sign of $P(b^n)$ is $(-1)^k$, thus proving~\eqref{e.sign_claim}.
\end{proof}

\begin{proof}[Proof of \cref{c.easy}]
Consider nonzero numbers $\xi_{1} > \dots > \xi_{d}$ whose sum is equal to $0$, and let $\boldsymbol{\xi} \coloneqq (\xi_1,\dots,\xi_d)$.
By \cref{l.number_theory}, there exists an Anosov linear automorphism $F_L \colon \T^d \to \T^d$ 
whose Lyapunov spectrum $\boldsymbol{\lambda}(L)$ is simple and has the same unstable index as $\boldsymbol{\xi}$.
If $n$ is sufficiently large then $\boldsymbol{\lambda}(L^n) = n \boldsymbol{\lambda}(L)$ strictly majorizes~$\boldsymbol{\xi}$.
Hence by \cref{t.majorize} there exists a conservative Anosov $C^\infty$ diffeomorphism $f \colon \T^d \to \T^d$ homotopic to $F_{L^n}$ with simple dominated splitting and such that $\boldsymbol{\lambda}(f) = \boldsymbol{\xi}$, as we wanted to prove.
\end{proof}

\subsection{Spectra of Anosov diffeomorphisms with simple dominated splitting on $\T^3$}\label{ss.fill_proof}

In this \lcnamecref{ss.fill_proof} we prove \cref{t.T3}.
Fix a hyperbolic matrix $L \in \GL(3,\Z)$ whose eigenvalues are all real and simple.
So the induced automorphism $F_L \colon \T^3 \to \T^3$ is Anosov and its Lyapunov spectrum is simple.
Let $u \in \{1,2\}$ be its unstable index.

\begin{proof}[Proof of the ``only if'' part of \cref{t.T3}]
Let $f\in \Diff_m^\infty(\T^3)$ be an Anosov diffeomorphism homotopic to the automorphism $F_L$, and admitting simple dominated splitting. Since $f$ and $F_L$ are topologically conjugate, 
they have the same unstable index $u$. 
Taking inverses if necessary, we can assume that $u=2$.
So the Lyapunov spectrum $\boldsymbol{\lambda}(f) = (\lambda_1(f),\lambda_2(f),\lambda_3(f))$ satisfies $\lambda_1(f)>\lambda_2(f)>0>\lambda_3(f)$. 
Let us show that  $\boldsymbol{\lambda}(f)$ is majorized by $\boldsymbol{\lambda}(L)$. 
As explained before, 
the inequality $\lambda_1(f)+\lambda_2(f) \le \lambda_1(L)+\lambda_2(L)$ is immediate: see \eqref{e.h_condition_proof}.
Therefore, we need to show that $\lambda_1(f) \le \lambda_1(L)$.

Let $\tilde{f}$ be a lift of $f$ to the universal covering $\R^3$. 
Since $f$ is homotopic to~$F_L$, we have $\tilde{f} = L + \phi$ for some $\Z^3$-periodic map $\phi \colon \R^3 \to \R^3$. 
So, for every $n \ge 0$, 
\[
\tilde{f}^n = L^n + \sum_{k=0}^{n-1} L^k \circ \phi \circ \tilde{f}^{n-1-k} \, .
\]
Since $\phi$ is bounded, it follows that there is a constant $C_1>0$ (independent of $n$) such that, for all $x$, $y \in \R^d$ with $\|x - y\| \le 1$ we have:
\[
\| \tilde{f}^n(x) - \tilde{f}^n(y) \| \le C_1 \sum_{k=0}^n \| L^k \| 
\]
Since the top eigenvalue of the linear map $L$ is simple, there is another constant $C_2>0$ such that $\| L^k \| \le C_2 e^{k \lambda_1(L)}$ for all $k \ge 0$.
In particular,
\[
\| \tilde{f}^n(x) - \tilde{f}^n(y) \| \le C_3 e^{n \lambda_1(L)} \, ,
\]
where $C_3>0$ is another constant.

By \cite[Theorem~1.3]{BBI} 
(see also \cite[Corollary 7.7]{Potrie}), the strong unstable foliation in the universal covering is quasi-isometric; this means that there is a constant $C_4>0$ such that if $I \subset \R^3$ is a segment of strong unstable manifold then its length, denoted by $\len(I)$, can be bounded as:
\[
\len(I) \le C_4\diam(I)+C_4.
\]
Hence, for every such a segment with $\diam(I) \le 1$, and every $n\geq 0$, 
\[
\len(\tilde{f}^n(I))\leq C_4 C_3 e^{n\lambda_1(L)} + C_4 \, .
\]
If follows from the next \lcnamecref{l.length_growth} that $\lambda_1(f) \le \lambda_1(L)$, as we wanted to prove. 
\end{proof}

\begin{lemma}\label{l.length_growth}
For each $x\in \T^3$, let $W_x \subset \T^3$ be the segment of strong unstable leaf of length $1$ for which $x$ is a midpoint. 
Then, for $m$-almost every $x\in \T^3$, 
\[
\limsup_{n\to\infty} \frac{\log\len(f^n(W_x))}{n} \ge \lambda_1(f) \, .
\]
\end{lemma}

(See \cite{SaghinXia} for related results.)

\begin{proof}
Let $\ell$ denote $1$-dimensional Hausdorff measure (i.e., length) on $\T^3$. 
For every $x\in \T^3$ we have:
\[
\ell(f^n(W_x)) = \int_{W_x} \| Df^n |_{E_1(y)} \| \dd \ell(y) \, .
\]
Since $\ell(W_x) = 1$, Jensen's inequality yields:
\[
\log \ell(f^n(W_x)) \ge \int_{W_x} \log \| Df^n |_{E_1(y)} \| \dd \ell(y) \, .
\]
Let $R$ be the set of points $y\in\T^3$ for which $\lim_{n\to \infty}\frac{1}{n} \log \| Df^n |_{E_1(y)} \| = \lambda_1(f)$; then $m(R)=1$.
By absolute continuity of the strong unstable foliation \cite[Lemma~10]{PesinSinai},
for $m$-almost every $x\in \T^3$ we have $\ell(W_x\cap R) = 1$.
Therefore,
\[
\frac{\log \ell(f^n(W_x))}{n} \ge
\int_{W_x \cap R} \frac{\log \| Df^n |_{E_1(y)} \|}{n} \dd \ell(y) \xrightarrow[n\to\infty]{} \lambda_1(f) \, . \qedhere
\]
\end{proof}

\begin{proof}[Proof of the ``if'' part of \cref{t.T3}]
Now we fix a vector $\boldsymbol{\xi} = (\xi_1,\xi_2,\xi_3)$ such that $\xi_1 > \xi_2 >\xi_3$, $\xi_u>0>\xi_{u+1}$, and $\boldsymbol{\xi} \preccurlyeq \boldsymbol{\lambda}(L)$. We want to find a smooth conservative Anosov diffeomorphism $f$ homotopic to $F_L$, admitting a simple dominated splitting, and with spectrum $\boldsymbol{\lambda}(f) = \boldsymbol{\xi}$.
If $\boldsymbol{\xi}$ is strictly majorized by $\boldsymbol{\lambda}(L)$ then the existence of $f$ is guaranteed by \cref{t.majorize}.
So let us assume that majorization is not strict, that is,
either $\xi_1=\lambda_1(L)$ or $\xi_1+\xi_2=\lambda_1(L)+\lambda_2(L)$.
We can assume that only one of these equalities is satisfied, since otherwise we can simply take $f=F_L$.

Let $E^L_1$, $E^L_2$, $E^L_3$ denote the eigenspaces of $L$ corresponding to the Lyapunov exponents $\lambda_1(L)$, $\lambda_2(L)$, $\lambda_3(L)$, respectively.
In case $\xi_1=\lambda_1(L)$ we shall perform the deformation in \cref{t.majorize} in such a way that the foliation $\cF^{23}$ parallel to $E^L_2\oplus E^L_3$ is preserved, while in the case $\xi_1+\xi_2=\lambda_1(L)+\lambda_2(L)$  (or equivalently $\lambda_3(f)=\lambda_3(L)$) we shall do it in such a way that foliation $\cF^{12}$ parallel to $E^L_1\oplus E^L_2$ is preserved. 
Since both cases are dealt with similarly, we will concentrate ourselves on the second case, namely $\xi_3 = \lambda_3(L)$.

The following observation will make the argument simpler:

\begin{lemma}\label{l.rig}
Let $f$ be a conservative Anosov diffeomorphism of $\T^3$ homotopic to $F_L$.
If $f$ preserves the foliation $\cF^{12}$ then $\lambda_3(f)=\lambda_3(L)$.\footnote{Actually, $\lambda_{3,\mu}(f)=\lambda_3(L)$ for any $f$-invariant probability measure $\mu$, by the same proof.}
\end{lemma}

\begin{proof} 
Let $\tilde{f}$ be a lift of $f$ to the universal covering $\R^3$. 
Since $f$ is homotopic to~$F_L$, we have $\tilde{f} = L + \phi$ for some $\Z^3$-periodic map $\phi$.
Let $P_1$, $P_2$, $P_3$ be the projections associated to the splitting $\R^3=E_1^L\oplus E_2^L\oplus E_3^L$. 
The fact that $f$ preserves the foliation $\cF^{12}$ implies that $\tilde{f}$ preserves the foliation $\widetilde{\cF}^{12}$ of $\R^3$ along planes parallel to $E_1^L\oplus E_2^L$.
For every $x\in\R^3$ and every $v\in E_1^L \oplus E_2^L$, we have $P_3\circ\tilde{f}(x) = P_3\circ\tilde{f}(x+v)$ and therefore $P_3\circ\phi(x) = P_3\circ\phi(x+v)$.
Since $F_L$ is a $3$-dimensional Anosov automorphism, 
the plane $E_2^L\oplus E_3^L$ projects to a dense subset of $\T^3$,
and since the map $P_3 \circ \phi$ is $\Z^3$-periodic, it must be constant. 
So the derivative of $f$ written w.r.t.\ to the splitting $\R^3=E_1^L\oplus E_2^L\oplus E_3^L$ is necessarily of the form
\[
\begin{pmatrix}
	\;\; * \;\; & \;\; * \;\; & \;\; 0 \;\; \\ 
	* & * & 0 \\ 
	* & * & \pm e^{\lambda_3(L)} 
\end{pmatrix}
\, .
\]
Since $\tilde{f}$ is volume preserving, this implies that the absolute value of the determinant of $D\tilde{f}$ restricted to $E_1^L\oplus E_2^L$ is everywhere constant $e^{-\lambda_3(L)}$. The same is true for $Df$ and hence
$\lambda_1(f) + \lambda_2(f) = -\lambda_3(L)$, that is, $\lambda_3(f) = \lambda_3(L)$.
\end{proof}

Coming back to the proof of \cref{t.T3}, we need a variation of the central proposition (\cref{p.central}) where all diffeomorphisms preserve the foliation $\cF^{12}$: 

\begin{proposition}\label{p.central_T3}
Let $a_1$, $\sigma$, and $\delta_0$ be positive numbers. Then, there exists $\delta \in (0,\delta_0)$ with the following properties:

Let $f\in \Diff_m^\infty(\T^3)$ be an Anosov diffeomorphism  homotopic to $F_L$ 
admitting a simple dominated splitting, and such that $\mathsf{g}_u(\boldsymbol{\lambda}(f)) \ge \sigma$.
In addition, assume that $f$ preserves the foliation $\cF^{12}$ (and so $\lambda_3(f) = \lambda_3(L)$, by \cref{l.rig}).
Then there exists a continuous map
\[
t \in [0,1] \mapsto f_t \in \Diff_m^\infty(\T^3)
\]
where $f_0 = f$
and for each $t \in [0,1]$, the conservative diffeomorphism $f_t$ is Anosov, admits a simple dominated splitting, and its top Lyapunov exponent satisfies: 
\begin{gather*}
\lambda_1(f) - 4\delta a_1 < \lambda_1(f_t) < \lambda_1(f) + \phantom{1}\delta a_1 \, , \\
\lambda_1(f) - 4\delta a_1 < \lambda_1(f_1) < \lambda_1(f) -           2\delta a_1 \, .
\end{gather*}
In addition, each $f_t$ preserves the foliation $\cF^{12}$ (and so $\lambda_3(f_t) = \lambda_3(L)$, by \cref{l.rig}).
\end{proposition}

Once this is established, we mimic the proof of \cref{t.majorize}; namely, we concatenate paths produced by the \lcnamecref{p.central_T3} and obtain a deformation of $F_L$ that ends with a diffeomorphism having Lyapunov spectrum equal to $\boldsymbol{\xi}$.

In order to prove \cref{p.central_T3}, we begin by modifying the construction of the Lyapunov charts in \cref{p.adapted_charts}. 
Since we are working in the torus $\T^3 = \R^3/\Z^3$, we can identify tangent spaces $T_x\T^3$ with $\R^3$.
Let $\pi \colon \R^3 \to \T^3$ be the quotient projection.
For each $x \in \T^3$, let $\cL_x \colon \R^3 \to \R^3$ be a linear map that takes the canonical basis $\{e_1,e_2,e_3\}$ of $\R^3$
to a basis $\{\cL_x(e_1), \cL_x(e_2), \cL_x(e_3)\}$ of $\R^3$ which is orthonormal for the Lyapunov metric $\biangle{\mathord{\cdot}, \mathord{\cdot}}_x$, and moreover $\cL_x(e_j) \in E_j^f(x)$ for each~$j$.
Then we define the Lyapunov charts 
\[
\Phi_x(z) \coloneqq x + \pi (\cL_x(z)) \, . 
\]
These charts have all properties from \cref{p.adapted_charts} and the following additional one:
\begin{enumerate}[start=5]
\item $\Phi_x$ is a \emph{foliated chart}, that is, if $z_1, z_2, z_3$ denote canonical coordinates in $\R^3$ then $\Phi_x$ maps   levels set of $z_3$ (horizontal slices) into leaves of~$\cF^{12}$. 
\end{enumerate}
Indeed, the derivatives $D\Phi_x(z)$ are constant and equal to $\cL_x$, and $\cL_x$ maps the plane $\R e_1 \oplus \R e_2$ to the plane $E_1^f(x) \oplus E_2^f(x) = E_1^L \oplus E_2^L$, which is tangent to $\cF^{12}$.

Then we follow the proof of \cref{p.central} (but with $a_2=0$). 
To summarize, the deformation $f_t$ of $f$ is constructed as follows:
\begin{itemize}
	\item We select a disjoint family of Lyapunov balls $\cB_1$, \dots, $\cB_p$; each ball $\cB_i$ equals $\Phi_i(\B)$, where $\B$ is the unit ball $\B$ in $\R^3$ and $\Psi_i (z) = \Phi_{x_i}(s_i z)$ is a rescaled Lyapunov chart;
	\item On each Lyapunov ball $\cB_i$ the deformation is defined as $f_t = f \circ g_t$, where $g_t = \Psi_i \circ h^{(1)}_{b_1 t} \circ \Psi_i^{-1}$; here $b_1>0$ is a constant and $h^{(1)}_t$ is the first elementary model deformation.
\end{itemize}
Inspecting the equations \eqref{e.def_R_theta}, \eqref{e.def_h_elementary} that define the diffeomorphism $h^{(1)}_t \colon \B \to \B$, we immediately see that it preserves horizontal slices (i.e.\ level sets of $z_3$).
Since $\Psi_i$ maps horizontal slices into leaves of $\cF^{12}$, the upshot is that each diffeomorphism $f_t$ preserves the foliation $\cF^{12}$. 
This proves \cref{p.central_T3}.
As explained before, \cref{t.T3} follows.
\end{proof}


\bigskip


\begin{thebibliography}{99}  


\MRbibitem{0358863}{Anosov}
\textsc{Anosov, D.V.} --
Existence of smooth ergodic flows on smooth manifolds. 
\textit{Izv.\ Akad.\ Nauk SSSR Ser.\ Mat.\ }38 (1974), 518--545 (Russian).
\textit{Math.\ USSR-Izv.\ }8 (1974), no.\ 3, 525--552. (English translation).



\MRbibitem{0370662}{AnosovK}
\textsc{Anosov, D.V.; Katok, A.B.} -- 
New examples in smooth ergodic theory. Ergodic diffeomorphisms. 
\textit{Trudy Moskov.\ Mat.\ Ob\v{s}\v{c}.\ }23 (1970), 3--36 (Russian).
\textit{Trans.\ Moscow Math.\ Soc.\ }23 (1970), 1--35 (English translation).



\MRbibitem{1723992}{LArnold}
\textsc{Arnold, L.} --
\textit{Random dynamical systems.}
Springer-Verlag, Berlin, 1998.


\MRbibitem{2736152}{Avila}
\textsc{Avila, A.} -- 
On the regularization of conservative maps. 
\textit{Acta Math.} 205 (2010), no.\ 1, 5--18. 

\MRbibitem{3578917}{ACW}
\textsc{Avila, A.; Crovisier, S.; Wilkinson, A.} --
Diffeomorphisms with positive metric entropy.
\textit{Publ.\ Math.\ IHES} 124 (2016), 319--347. 

\MRbibitem{2032482}{BarBon}
\textsc{Baraviera, A.T.; Bonatti, C.} --
Removing zero Lyapunov exponents. 
\textit{Ergodic Theory Dynam.\ Systems} 23 (2003), no.\ 6, 1655--1670.

\MRbibitem{2348606}{BP}
\textsc{Barreira, L.; Pesin, Ya.} --
\textit{Nonuniform hyperbolicity. Dynamics of systems with nonzero Lyapunov exponents.}
Encyclopedia of Mathematics and its Applications, 115. 
Cambridge University Press, Cambridge, 2007.


\MRbibitem{4084180}{BErch1}
\textsc{Barthelm\'e, T.; Erchenko, A.} -- 
Flexibility of geometrical and dynamical data in fixed conformal classes.
\textit{Indiana Univ.\ Math.\ J.\ }69 (2020), no.\ 2, 517--544. 


\bibitem{BErch2}
\textsc{Barthelm\'e, T.; Erchenko, A.} -- 
Geometry and entropies in a fixed conformal class on surfaces.
\arxiv{1902.02896}.
\textit{Ann.\ Inst.\ Fourier}, to appear. 

\MRbibitem{1944399}{Bochi_ETDS} 
\textsc{Bochi, J.} --
Genericity of zero Lyapunov exponents.  
\textit{Ergodic Theory Dynam. Systems} 22 (2002), no. 6, 1667--1696. 

\MRbibitem{3966831}{Bochi_ICM}
\textsc{Bochi, J.} --
Ergodic optimization of Birkhoff averages and Lyapunov exponents.
\textit{Proc.\ Int.\ Cong.\ of Math. -- 2018 Rio de Janeiro}, Vol.\ 2, 1821--1842.

\MRbibitem{2948788}{BoBo}
\textsc{Bochi, J.; Bonatti, C.} -- 
Perturbation of the Lyapunov spectra of periodic orbits. 
\textit{Proc.\ Lond.\ Math.\ Soc.\ }105 (2012), no.\ 1, 1--48. 

\MRbibitem{2180404}{BochiViana}
\textsc{Bochi, J.; Viana, M.} --
The Lyapunov exponents of generic volume-preserving and symplectic maps.
\textit{Ann.\ of Math.\ }161 (2005), no.\ 3, 1423--1485.

\MRbibitem{2105774}{BDV}
\textsc{Bonatti, C.; D\'iaz, L.J.; Viana, M.} -- 
\textit{Dynamics beyond uniform hyperbolicity. A global geometric and probabilistic perspective.} 
Encyclopaedia of Mathematical Sciences, 102. 
Mathematical Physics, III. Springer-Verlag, Berlin, 2005.


\MRbibitem{2180404}{BBI}
\textsc{Brin, M.; Burago, D.; Ivanov, S.}, --
Dynamical coherence of partially hyperbolic diffeomorphisms of the $3$-torus.
\textit{J.\ Mod.\ Dyn.\ }3 (2009), no.\ 1, 1--11.




\MRbibitem{3798715}{Butler1}
\textsc{Butler, C.} -- 
Rigidity of equality of Lyapunov exponents for geodesic flows. 
\textit{J.\ Differential Geom.\ }109 (2018), no.\ 1, 39--79. 

\bibitem{Butler2}
\textsc{Butler, C.} -- 
Characterizing symmetric spaces by their Lyapunov spectra.
\arxiv{1709.08066}


\MRbibitem{1898798}{DolgoPesin}
\textsc{Dolgopyat, D.; Pesin, Ya.} --
Every compact manifold carries a completely hyperbolic diffeomorphism.
\textit{Ergodic Theory Dynam.\ Systems} 22 (2002), no.\ 2, 409--435.



\MRbibitem{3927515}{Erch}
\textsc{Erchenko, A.} -- 
Flexibility of Lyapunov exponents for expanding circle maps.
\textit{Discrete Contin.\ Dynam.\ Systems} 39 (2019), no.\ 5, 2325--2342.


\MRbibitem{3990954}{ErchK}
\textsc{Erchenko, A.; Katok, A.} -- 
Flexibility of entropies for surfaces of negative curvature.
\textit{Israel J.\ Math.\ }232 (2019), no.\ 2, 631--676.


\MRbibitem{3188624}{FG}
\textsc{Farrell, F.~T.; Gogolev, A.} --
The space of Anosov diffeomorphisms.
\textit{J.\ Lond.\ Math.\ Soc.\ }89 (2014), no.\ 2, 383--396. 

\MRbibitem{4145804}{GKS}
\textsc{Gogolev, A.; Kalinin, B.; Sadovskaya, V.} --
Local rigidity of Lyapunov spectrum for toral automorphisms.
\textit{Israel J.\ Math.\ }238 (2020), no.\ 1, 389--403. 

\MRbibitem{2371598}{Gourmelon}
\textsc{Gourmelon, N.} --
Adapted metrics for dominated splittings.
\textit{Ergodic Theory Dynam.\ Systems} 27 (2007), no.\ 6, 1839--1849. 


\MRbibitem{3661819}{HuJJ_older}
\textsc{Hu, H.; Jiang, M.; Jiang, Y.} -- 
Infimum of the metric entropy of hyperbolic attractors with respect to the SRB measure.
\textit{Discrete Contin.\ Dyn.\ Syst.\ }22 (2008), no.\ 1-2, 215--234.

\MRbibitem{3661819}{HuJJ}
\textsc{Hu, H.; Jiang, M.; Jiang, Y.} -- 
Infimum of the metric entropy of volume preserving Anosov systems.
\textit{Discrete Contin.\ Dyn.\ Syst.\ }37 (2017), no.\ 9, 4767--4783. 

\MRbibitem{0554383}{K79}
\textsc{Katok, A.} --
Bernoulli diffeomorphisms on surfaces.
\textit{Ann.\ of Math.\ }110 (1979), no.\ 3, 529--547.


\MRbibitem{0721728}{K82} 
\textsc{Katok, A.} --
Entropy and closed geodesics. 
\textit{Ergodic Theory Dynam.\ Systems} 2 (1982), no.\ 3--4, 339--365.



\MRbibitem{1326374}{KH}
\textsc{Katok, A.; Hasselblatt, B.} --
\textit{Introduction to the modern theory of dynamical systems.} 
With a supplementary chapter by Katok and Leonardo Mendoza. 
Encyclopedia of Mathematics and its Applications, 54. 
Cambridge University Press, Cambridge, 1995.



\MRbibitem{3503686}{KRH}
\textsc{Katok, A.; Rodriguez Hertz, F.} --
Arithmeticity and topology of smooth actions of higher rank abelian groups.
\textit{J.\ Mod.\ Dyn.\ }10 (2016), 135--172. 

\MRbibitem{1336823}{Koba}
\textsc{Kobayashi, S.} --
\textit{Transformation groups in differential geometry.}
Reprint of the 1972 edition. Classics in Mathematics. Springer-Verlag, Berlin, 1995. 

\MRbibitem{4157477}{LMY}
\textsc{Liang, C; Marin, K.; Yang, J.} -- 
$C^1$-openness of non-uniform hyperbolic diffeomorphisms with bounded $C^2$-norm.
\textit{Ergodic Theory Dynam.\ Systems} 40 (2020), no.\ 11, 3078--3104. 

\MRbibitem{2759813}{MOA}
\textsc{Marshall, A.W.; Olkin, I.; Arnold, B.C.} -- 
\textit{Inequalities: theory of majorization and its applications.} 2nd ed. 
Springer Series in Statistics. Springer, New York, 2011.

\MRbibitem{3951424}{MicenaTah}
\textsc{Micena, F.; Tahzibi, A.} --
A note on rigidity of Anosov diffeomorphisms of the three torus.
\textit{Proc.\ Amer.\ Math.\ Soc.\ }147 (2019), 2453--2463.

\MRbibitem{0182927}{Moser}
\textsc{Moser, J.} --
On the volume elements on a manifold.
\textit{Trans.\ Amer.\ Math.\ Soc.\ }120 (1965), 286--294.

\MRbibitem{0986792}{Newhouse}
\textsc{Newhouse, S.E.} -- 
Continuity properties of entropy. 
\textit{Ann.\ of Math.\ }129 (1989), no.\ 2, 215--235. 

\MRbibitem{0721733}{PesinSinai}
\textsc{Pesin, Ya.~B.; Sinai, Ya.~G.} --  
Gibbs measures for partially hyperbolic attractors. 
\textit{Ergodic Theory Dynam. Systems} 2 (1982), no. 3-4, 417--438 (1983).

\MRbibitem{3223375}{PonceTah}
\textsc{Ponce, G.; Tahzibi, A.} -- 
Central Lyapunov exponent of partially hyperbolic diffeomorphisms of $\mathbb{T}^3$.
\textit{Proc.\ Amer.\ Math.\ Soc.\ }142 (2014), no.\ 9, 3193--3205. 


\MRbibitem{2180404}{Potrie}
\textsc{Potrie, R.} --
Partial hyperbolicity and foliations in {$\Bbb{T}^3$}.
\textit{J.\ Mod.\ Dyn.\ }9 (2015), 81--121.  

\MRbibitem{2504049}{SaghinXia}
\textsc{Saghin, R.; Xia, Z.} -- 
Geometric expansion, Lyapunov exponents and foliations.
\textit{Ann.\ Inst.\ H.\ Poincar\'{e} Anal.\ Non Lin\'{e}aire} 26 (2009), no.\ 2, 689--704. 


\MRbibitem{3995224}{SaghinYang}
\textsc{Saghin, R.; Yang, J.} --
Lyapunov exponents and rigidity of Anosov automorphisms and skew products.
\textit{Adv.\ Math.\ }355 (2019), 106764, 45 pp. 

\MRbibitem{3457601}{Samba}
\textsc{Sambarino, M.} -- 
A (short) survey on dominated splittings. 
\textit{Mathematical Congress of the Americas}, 149--183, 
Contemp.\ Math., 656, Amer.\ Math.\ Soc., Providence, RI, 2016. 

\MRbibitem{1738057}{ShubWilk}
\textsc{Shub, M.; Wilkinson, A.} --
Pathological foliations and removable zero exponents. 
\textit{Invent.\ Math.\ }139 (2000), no.\ 3, 495--508. 

\MRbibitem{0006217}{Vija}
\textsc{Vijayaraghavan, T.} -- 
On the fractional parts of the powers of a number. II. 
\textit{Proc.\ Cambridge Philos.\ Soc.\ }37 (1941), 349--357. 


\MRbibitem{0889979}{Yomdin}
\textsc{Yomdin, Y.} -- 
Volume growth and entropy. 
\textit{Israel J.\ Math.\ }57 (1987), no.\ 3, 285--300. 

\end{thebibliography}
\end{document}